\documentclass{amsart}
\usepackage{amsmath,amscd,amssymb,amsthm}
\usepackage{amsfonts}
\usepackage[all]{xy}
\usepackage{mathtools}
\usepackage{enumerate}

\usepackage{hyperref}

\usepackage{mathrsfs}

\usepackage{graphicx}

\newcommand{\into}{\hookrightarrow}

\newcommand{\xto}{\xrightarrow}

\newcommand{\uoplus}{\sqcap}

\newcommand{\DGA}{\mathbf{DGA}}

\newcommand{\F}{\mathbb{F}}
\newcommand{\FF}{F}

\newcommand{\Z}{\mathbb{Z}}

\newcommand{\BB}{\mathbf{B}} 

\newcommand{\SB}{\mathbf{S}} 

\newcommand{\ff}{\mathfrak{f}} 

\newcommand{\mm}{m}

\newcommand{\Xh}{\mathcal{X}}

\newcommand{\Hom}{\mathrm{Hom}}

\newcommand{\uHom}{\underline{\Hom}}

\newcommand{\id}{\mathrm{id}}

\newcommand{\op}{\mathrm{op}}

\newcommand{\HH}{\mathrm{HH}}

\theoremstyle{plain}
\newtheorem{thm}[subsection]{Theorem}

\newtheorem{prop}[subsection]{Proposition}
\newtheorem{lemma}[subsection]{Lemma}
\newtheorem{cor}[subsection]{Corollary}

\theoremstyle{definition}
\newtheorem{defn}[subsection]{Definition}

\newtheorem{examples}[subsection]{Examples}
\newtheorem{notn}[subsection]{Notation}

\theoremstyle{remark}
\newtheorem{rem}[subsection]{Remark}

\begin{document}
\title{Real projective groups are formal}
\author{Ambrus P\'al}
\address{Department of Mathematics, 180 Queen's Gate, Imperial College, London, SW7 2AZ, United Kingdom}
\email{a.pal@imperial.ac.uk}

\author{Gereon Quick}
\address{Department of Mathematical Sciences, NTNU, NO-7491 Trondheim, Norway}
\email{gereon.quick@ntnu.no}

\begin{abstract} 
We prove that the mod $2$ cohomology algebras of real projective groups are intrinsically formal. As a consequence we derive the Hopkins--Wickelgren formality and the strong Massey vanishing conjecture for a large class of fields.   
\end{abstract}

\maketitle

\section{Introduction}

Let $G$ be a profinite group. An {\it embedding problem for $G$} is a solid diagram:
\[
\xymatrix{
 & G \ar[d]^-{\phi}\ar@{.>}[ld]_-{\widetilde{\phi}}
\\
\widetilde{H} \ar[r]_-{\alpha} & H,}
\]
where $H$ and $\widetilde{H}$ are finite groups, the solid arrows are continuous homomorphisms and $\alpha$ is surjective. 
A {\it solution} of this embedding problem is a continuous homomorphism $\widetilde{\phi} \colon G\to \widetilde{H}$ which makes the diagram commutative.  
We say that the embedding problem above is {\it real} if for every involution $t\in G$ with $\phi(t)\neq 1$ there is an involution $\tilde{h} \in \widetilde{H}$ with $\alpha(\tilde{h})=\phi(t)$, i.e., if involutions do not provide an obstruction for the existence of a solution. 

\begin{defn} 
Following Haran and Jarden \cite{HJ} we say that a profinite group $G$ is {\it real projective} if $G$ has an open subgroup without $2$-torsion, and if every real embedding problem for $G$ has a solution. 
\end{defn}


For a prime number $p$ and a profinite group $G$, let $C^*(G,\mathbb F_p)$ denote the differential graded $\mathbb F_p$-algebra of continuous cochains of $G$ with values in
$\mathbb F_p$.  
The aim of this paper is to show that the differential graded algebra $C^*(G,\mathbb F_p)$ of a real projective group $G$ has the following strong property.  

\begin{defn} 
Let $\F$ be a field and let $C^*$ be a differential graded 
$\F$-algebra with cohomology $H^*$. 
Then $C^*$ is called {\it formal} if there is a sequence $C^* \leftarrow T^* \rightarrow H^*$ 
of quasi-isomorphisms of differential graded algebras between $C^*$ and $H^*$ where we consider $H^*$ as a differential graded $\F$-algebra with trivial differential. 
Moreover, a graded $\F$-algebra $A$ is called \emph{intrinsically formal} 
if every differential graded algebra $C^*$ with $H^*(C^*)\cong A$ is formal. 
\end{defn}


The main result of this paper is the following: 

\begin{thm}\label{thm:mainthmintro}
Let $G$ be a real projective profinite group. 
Then the graded algebra $H^*(G,\mathbb F_2)$ is intrinsically formal. 
In particular, the differential graded $\F_2$-algebra  $C^*(G,\mathbb F_2)$ is formal. 
\end{thm}

\begin{rem}
For an odd prime $p$ and $G$ real projective, the $\F_p$-algebra $H^*(G,\mathbb F_p)$ is concentrated in degrees $0$ and $1$, i.e., $H^i(G,\F_p)=0$ for $i\ge 2$. 
The intrinsic formality of $H^*(G,\mathbb F_p)$ then follows easily from the analog of Theorem \ref{thm:KadeishvilicritST} for an odd prime $p$ (the result holds for augmented graded algebras over any field, see  \cite[Theorem 4.7, page 85]{Seidel-Thomas}) and the argument of the proof of Lemma \ref{lemma:dualformal} below. 
Hence, together with Theorem \ref{thm:mainthmintro}, we can conclude that $H^*(G,\mathbb F_p)$ is intrinsically formal and therefore $C^*(G,\mathbb F_p)$ is formal for all primes $p$. 
\end{rem}

Our proof of Theorem \ref{thm:mainthmintro} relies on a theorem due to Kadeishvili which states that a positively graded algebra $A$ over a field is intrinsically formal 
if the Hochschild cohomology groups $\HH^{n,2-n}(A,A)$ vanish for all $n\ge 3$.  
In order to show this vanishing for $A=H^*(G,\mathbb F_2)$ we deduce from Scheiderer's work in \cite{Sch} the description of $H^*(G,\mathbb F_2)$ for a real projective group $G$ as the following type of quadratic algebra: 
First, let $B$ be a Boolean ring, i.e., $x^2=x$ for every $x\in B$, and let $B_*=\bigoplus_{j\ge 0}B_j$ be its associated Boolean graded algebra, i.e., 
the $\F_2$-algebra with $B_0=\mathbb F_2$ and, for every positive integer $j$, $B_j=B$. 
Second, for an $\F_2$-vector space $V$, let $V_*=\bigoplus_{j\ge 0}V_j$ denote the graded algebra with $V_0=\F_2$, $V_1=V$ and $V_j=0$ for $j\geq2$. 
We refer to Section \ref{section:def_Boolean_and_dual} and Definitions \ref{def:boolean} and \ref{defn:dual} for more details. 
We show that $H^*(G,\mathbb F_2)$ is a quadratic algebra given as the connected sum of quadratic algebras $B_* \sqcap V_*$ (see Definition \ref{def:connected_sum}) 
where $V$ is a suitable vector space and $B$ is the Boolean ring of continuous functions from the space of conjugacy classes of involutions in $G$ to $\F_2$. 
Then we calculate the Hochschild cohomology groups of such an algebra by an explicit description of the cocycles and coboundaries in the case of a locally finite Boolean graded algebra in Theorem \ref{thm:maintheoremfinite}. 
Finally, we use a spectral sequence argument and the vanishing of the  higher derived limits of certain non-vanishing Hochschild cohomology groups to deduce the general case. Our main technical result is the following, see Theorem \ref{main}: 

\begin{thm}\label{thm:main_HH_intro} 
Let $B$ be a Boolean ring with $|B|\ge 8$, and let $V$ be an $\F_2$-vector space. 
Then 
\[
\HH^{k,\mm}(B_*\uoplus V_*,B_* \uoplus V_*)=0
\]
for every $m< 0$ and $k \ge 0$ such that $k\neq 1-\mm$.  
\end{thm}

\begin{rem}
In fact, we thereby prove that for every profinite group $G$ such that $H^*(G,\mathbb F_2)$ is of the form $B_*\uoplus V_*$ as in Theorem \ref{thm:main_HH_intro}, 
the graded algebra $H^*(G,\mathbb F_2)$ is intrinsically formal. 
However, real projective groups are the main examples of groups with this type of cohomology algebra that we are aware of. 
\end{rem}

One of the main consequences of formality of a differential graded associative algebra 
is the vanishing of Massey products. 
We will formulate a stronger form of the latter next.

\begin{defn}\label{11.1} 
Let $\F$ be a field and let $(C^*,\cup,\delta)$ be a differential graded associative $\F$-algebra with differential $\delta$. 
For an element $a\in C^d$ we write $\bar{a}:=(-1)^{1+d}a$.  
For an integer $n\geq2$, let $a_1,a_2,\ldots,a_n$ be a set of cohomology classes with $a_i \in H^{d_i}$. 
A \emph{defining system} for the $n$-fold Massey product of $a_1,a_2,\ldots,a_n$ is a set $\{a_{i,j}\}$ of elements of $C^{d_{i,j}}$ for $1\leq i<j\leq n+1$ and $(i,j)\neq(1,n+1)$ such that
\[
\delta(a_{i,j})=\sum_{k=i+1}^{j-1}\bar{a}_{i,k}\cup a_{k,j}
\]
and $a_1,a_2,\ldots,a_n$ is represented by $a_{1,2},a_{2,3},\ldots,a_{n,n+1}$. 
The degree $d_{i,j}$ of $a_{i,j}$ satisfies  
\begin{align*}
d_{i,j} = \sum_{s=i}^{j-1}d_s -(j-1-i) ~ \text{for}~i+1\le j \le n+1.
\end{align*}
We say that the \emph{$n$-fold Massey product} of $a_1,a_2,\ldots,a_n$ is \emph{defined} if there exists a defining system. 
The $n$-fold Massey product $\langle a_1,a_2,\ldots,a_n\rangle_{\{a_{i,j}\}}$ of $a_1,a_2,\ldots,a_n$ with respect to the defining system $\{a_{i,j}\}$ is the cohomology class of
\[
\sum_{k=2}^n \bar{a}_{1,k}\cup a_{k,n+1}
\]
in $H^{d_{n+1}}$ where $d_{n+1}=\sum_{i=1}^n d_i-n+2$. 
Let $\langle a_1,a_2,\ldots,a_n\rangle$ denote the subset of $H^{d_{n+1}}$ consisting of the $n$-fold Massey products of $a_1,a_2,\ldots,a_n$ with respect to all defining systems. 
We say that the $n$-fold Massey product of $a_1,a_2,\ldots,a_n$ {\it vanishes} if $\langle a_1,a_2,\ldots,a_n\rangle$ contains zero. \\
We say that $C^*$ {\it satisfies strong Massey vanishing} if for very $a_1,a_2,\ldots,a_n$ as above the following assertion holds: 
if 
\begin{equation}\label{1.6.1}
a_1\cup a_2=a_2\cup a_3=\cdots=a_{n-1}\cup a_n=0,
\end{equation}
then the $n$-fold Massey product of $a_1,a_2,\ldots,a_n$ vanishes. 
When condition \eqref{1.6.1} above is satisfied we say that {\it all neighbouring cup products vanish for the $n$-tuple $a_1,a_2,\ldots,a_n$}.
\end{defn}

For real projective groups, strong Massey vanishing for cohomology classes in degree one was proven in \cite[Theorem 1.7]{PSz}. 
For pro $p$-groups of elementary type, strong Massey vanishing for cohomology classes in degree one was proven in \cite[Theorem 1.2]{Quadrelli}. 
(See also the comments below Theorem \ref{thm:introvcdoneformal}.) 
As we explain in section \ref{section:Massey_products}, Theorem \ref{thm:mainthmintro} implies the following result: 

\begin{thm}\label{thm:strongMasseyvanishing_intro}
Let $G$ be a real projective profinite group, and let $p$ be a prime number. 
Then $C^*(G,\mathbb F_p)$ satisfies strong Massey vanishing for cohomology classes in arbitrary degrees. 
\end{thm}

\begin{rem}
For an odd prime $p$, $C^*(G,\mathbb F_p)$ satisfies strong Massey vanishing since $H^i(G,\F_p)=0$ for $i \ge 2$. 
Hence the difficult case is $p=2$. 
We emphasise that Theorem \ref{thm:strongMasseyvanishing_intro} proves the strong Massey vanishing conjecture for real projective groups for cohomology classes in arbitrary degrees. 
We thereby provide the first nontrivial class of profinite groups for which strong Massey vanishing holds beyond degree one. 
\end{rem}

Our interest in the formality of profinite groups originates from the arithmetic of fields. 
Let $\FF$ be any field whose characteristic is not two, and let $\Gamma(\FF)$ denote its absolute Galois group. 
In \cite[Question 1.4 on page 1306]{HW} Hopkins and Wickelgren raised the question whether $C^*(\Gamma(\FF),\mathbb F_2)$ is formal. 
We know that the answer is no in general. 
In fact, Positselski proved in \cite[Section 9.11]{PoMoscow} that, for an odd prime number $p$ and a finite extension $\FF$ of $\mathbb Q_p$ which contains a primitive $p$-root of unity, the differential graded algebra $C^*(\Gamma(\FF),\mathbb F_2)$ is not formal (see also \cite[\S 6]{Po}). 
We note that there is a counterexample to the strong Massey vanishing for fourfold products over $\mathbb Q$, due to Wittenberg--Harpaz, in \cite[Example A.15]{GMT}. 
This implies that the absolute Galois group of $\mathbb Q$ is not formal. 
In the more recent paper \cite{MS} Merkurjev--Scavia provide a counterexample over a field containing the algebraic closure of 
$\mathbb Q$, so even a more restricted prediction of Positselski on the formality of fields is false, see \cite[Question 1.5 and Theorem 1.6]{MS} and \cite{MS2023}. 
However, we still expect that many fields have formal absolute Galois groups. 
Our result is actually the first affirmative contribution to this conjecture as we will explain next.

Recall that a field $\FF$ has virtual cohomological dimension $\leq 1$ if there is a finite 
extension $K/\FF$ with $\mathrm{cd}(\Gamma(K))\leq 1$. 
As we recall in section \ref{section:Galois_cohomology}, 
it is a consequence of classical Artin--Scheier theory and Haran's 
work that the absolute Galois group of a field with virtual cohomological dimension $\leq 1$ is real projective. 
Hence our results for real projective groups imply the following:

\begin{thm}\label{thm:introvcdoneformal}
Let $\FF$ be a field with virtual cohomological dimension $\leq 1$, and let $p$ be a prime number. 
Then $H^*(\FF,\mathbb F_p)$ is intrinsically formal, and hence $C^*(\Gamma(\FF),\F_p)$ is formal and satisfies strong Massey vanishing. 
\end{thm}

This theorem is also a new contribution to the Massey vanishing conjecture by Min\'a\v c and T\^an \cite[Conjecture 1.1]{MT0}, a very active field of research (see e.g.~\cite{EM17}, \cite{GMT}, \cite{HW}, \cite{MS}, \cite{MT1}, \cite{PSz}, \cite{Quadrelli}). 
In fact, our theorem shows Massey vanishing in arbitrary cohomology degrees for fields with virtual cohomological dimension $\leq 1$. 
This is a much stronger result than in the previously known cases where only the degree one Massey products were considered. \\

Another consequence of our methods is a positive case of a conjecture by Positselski and Voevodsky on the Koszulity of Galois cohomology \cite[\S 0.1, page 128]{PoGalois}. 
In \cite{MPQT}, other positive cases of this conjecture were proven. 
Recall that a positively graded algebra $A$ over $\F_p$  is called {\it Koszul} if the groups $H_{i,j}(A) = \mathrm{Tor}^A_{i,j}(\F_p,\F_p)$ vanish for all $i \neq j$, 
where the first grading $i$ is the homological grading and the second grading $j$ is the internal grading which is induced from the grading of $A$. 
Based on Scheiderer's structure theorem for real projective groups we prove:

\begin{thm}\label{thm:intro_Koszul}
Let $\FF$ be a field with virtual cohomological dimension $\leq 1$, and let $p$ be a prime number. 
Then the cohomology algebra $H^*(\FF,\F_p)$ is Koszul. 
\end{thm}

\subsection*{Contents} 
In Section \ref{section:Galois_cohomology} we review the work of Haran--Jarden who showed that every real projective group arises as an absolute Galois group of a pseudo real closed field and deduce the structure of the mod $2$ cohomology algebra of real projective groups from Scheiderer's local-global principle for real projective groups. 
In Section \ref{section:Massey_products} we show that formality implies strong Massey vanishing for differential graded algebras. 
In Section \ref{section:HH} we review a Kadeishvili's sufficient condition for intrinsic formality in terms of graded Hochschild cohomology. 
In Section \ref{section:Koszul} we recall the Koszul complex of quadratic algebras and the notion of Koszul algebras. 
In Section \ref{section:def_Boolean_and_dual} we introduce Boolean and dual graded algebras and show that their connected sum is Koszul. 
In Section \ref{section:Boolean_finite} we prove our main technical result on the graded Hochschild cohomology of the connected sum of dual and Boolean graded algebras, first for finite Boolean subrings. 
In Section \ref{section:ML} we recall some facts needed on Mittag-Leffler functors. 
In Section \ref{section:Boolean_infinite} we extend the arguments from Section \ref{section:Boolean_finite} to the case of an infinite Boolean ring 
and thereby prove the main technical result Theorem \ref{thm:main_HH_intro}.   
In Section \ref{section:Main_result} we combine the results of the previous sections to prove Theorem \ref{thm:mainthmintro} and its consequences.   
In Appendix \ref{app:Boolean} we provide proofs of the assertions on Boolean rings we use in Sections \ref{section:def_Boolean_and_dual}, \ref{section:Boolean_finite} and \ref{section:Boolean_infinite}. 
In Appendix \ref{section:appendix} we prove the existence of the spectral sequence that we use in Section \ref{section:Boolean_infinite}.

\subsection*{Acknowledgement} 
The first-named author wishes to thank the hospitality of the Departments of Mathematical Sciences at University of Oslo and NTNU in Trondheim. 
Both authors gratefully acknowledge the partial support by the RCN Frontier Research Group Project No.\,250399 {\it Motivic Hopf equations}, the RCN Project No.\,313472 {\it Equations in Motivic Homotopy}, and the project \emph{Pure Mathematics in Norway} funded by the Trond Mohn Foundation.  
We thank the anonymous referee for many comments and suggestions which helped improving the paper.


\section{Fields whose absolute Galois group is real projective and Scheiderer's theorem on their Galois cohomology}\label{section:Galois_cohomology}

In this section, we first review the work of Haran--Jarden who showed that every real projective group arises as an absolute Galois group of a pseudo real closed field. 
Then we deduce from Scheiderer's local-global principle the structure of the mod $2$ cohomology algebra of real projective groups.  

For a profinite group $G$, let $\mathrm{cd}(G)$ denote its cohomological dimension as defined in \cite{Se}. 
Haran gave the following very useful characterisation of real projective groups \cite[Proposition 2.2 on page 226]{Ha}: 

\begin{thm}[Haran]\label{haran} 
A profinite group $G$ is real projective if and only if $G$ has an open subgroup $G_0$ of index $\leq 2$ with $\mathrm{cd}(G_0)\leq 1$, and every involution $t\in G$ is self-centralising, that is, we have $C_G(t)=\{1, t\}$. 
\end{thm}

Haran's theorem immediately implies that every closed subgroup of a real projective group $G$ is also real projective, and in particular every torsion element of $G$ has order dividing $2$. 

For a field $\FF$, let $\Gamma(\FF)$ denote its absolute Galois group,  
and we write $\mathrm{cd}(\FF) = \mathrm{cd}(\Gamma(\FF)$.  
Recall that a field $\FF$ has virtual cohomological dimension $\leq 1$ if there is a finite 
extension $K/\FF$ with $\mathrm{cd}(K)\leq 1$.  
Since the only torsion elements in the absolute Galois group of $\FF$ are the involutions coming from the orderings of $\FF$, it is equivalent (by a theorem in \cite{Se2}) to require $\mathrm{cd}(K)\leq 1$ for any fixed finite separable extension $K$ of $\FF$ without orderings, for example for $K = \FF(\mathbf i)$, where $\mathbf i =\sqrt{-1}$. 
In particular, if $\FF$ itself cannot be ordered (which is equivalent to $-1$ being a sum of squares in $\FF$), this condition is equivalent to $\mathrm{cd}(\FF)\leq 1$.

\begin{examples} 
Examples of fields $\FF$ which can be ordered with $\mathrm{cd}(\FF(\mathbf i))\leq 1$ include real closed fields, function fields in one variable over any real closed ground field, the field of Laurent series in one variable over any real closed ground field, and the field $\mathbb Q^{ab}\cap\mathbb R$ which is the subfield of $\mathbb R$ generated by the numbers $\cos(\frac{2\pi}{n})$, where $n\in\mathbb N$.
\end{examples}

The following is a well-known consequence of Artin--Schreier theory. 

\begin{lemma}\label{central} 
Let $\Gamma$ be an absolute Galois group and let $x\in\Gamma$ be an involution. 
Then the centraliser $C_{\Gamma}(x)$ is the group $C=\{1,x\}$ generated by $x$. 
\end{lemma}
\begin{proof} Let $K$ be a field whose absolute Galois group is isomorphic to $\Gamma$ and let $\overline K$ be the separable closure of $K$. 
Let $R\subset\overline K$ be the fixed field of $C$. 
Then the absolute Galois group of $R$ is $\mathrm{Gal}(\overline K/R)=C \cong \mathbb Z/2$, so $R$ is real closed by \cite[Proposition 2.4, page 452]{La}. 
In particular, $R$ has characteristic zero, so the same holds for all subfields of $\overline K$. 
Moreover, the latter is algebraically closed. 
Since $C_{\Gamma}(x)$ is the pre-image of $1$ under the map $y\mapsto
x^{-1}y^{-1}xy$, it is closed. Let $F\subset\overline K$ be the fixed field of
$C_{\Gamma}(x)$. Since $C\subseteq C_{\Gamma}(x)$ we get that $F$ is a subfield of $R$. Note that $C_{\Gamma}(x)$ leaves $R$ invariant; indeed if $\alpha\in R$, then $x\alpha=\alpha$, and hence for every $y\in C_{\Gamma}(x)$ we have $xy\alpha=yx\alpha=y\alpha$, so $y\alpha\in R$.
The extension $R/F$ is algebraic, so $R$ is the real closure of $F$. 
Therefore, the Galois group $\mathrm{Gal}(R/F)$ is trivial by \cite[Theorem 2.9, page 455]{La}. 
We get that $C_{\Gamma}(x)$ lies in the subgroup of $\Gamma$ fixing $R$, so by the profinite version of the Galois correspondence $C_{\Gamma}(x)\subseteq C$. 
\end{proof}

This shows that Theorem \ref{haran} has the following consequence:  

\begin{cor}\label{cor:vcd_one}
The absolute Galois group of a field $F$ is real projective if and only if $F$ satisfies $\mathrm{cd}(F(\mathbf i))\leq 1$. \qed
\end{cor}

\begin{defn} 
Let $\FF$ be a field. 
The {\it real spectrum} $\mathrm{Spr}(F)$ of $F$ is the set of all orderings of $F$. 
We consider $\mathrm{Spr}(F)$ as a topological space by equipping it with the {\it Harrison topology}. 
The latter is defined as the topology on $\mathrm{Spr}(F)$ with sub-basis consisting of the sets 
$H(a)$ of all orderings on $F$ in which $a \in F$ is positive. 
By \cite[Chapter VIII, \S 6, Theorem 6.3]{Lam05}, the Harrison topology makes $\mathrm{Spr}(F)$ a compact and totally disconnected topological space. 
\end{defn}

\begin{defn} 
For every ordering $< \in \mathrm{Spr}(F)$, let $F_<$ denote the real closure of the ordered field $(\FF,<)$. 
We say that the field $F$ is {\it pseudo real closed} if every absolutely irreducible variety defined over $F$ which has a simple $F_<$-rational point for every $<\in\mathrm{Spr}(F)$ has an $F$-rational point.
\end{defn}

\begin{rem} 
By the work of Haran--Jarden \cite[Theorem 10.4 on page 487]{HJ}, the absolute Galois group of a pseudo real closed field is real projective. 
We can also deduce this result from the previous discussion as follows. 
Recall that a field $F$ is {\it pseudo algebraically closed} if every geometrically irreducible variety defined over $F$ has an $F$-rational point. 
By a theorem of Ax \cite{Ax}, the absolute Galois group $\Gamma$ of such a field is projective, and in particular it has cohomological dimension at most one. 
By Prestel's Extension Theorem \cite[Theorem 3.1, page 148]{Prestel}, every algebraic extension of a pseudo real closed field is pseudo real closed. 
This implies that if $F$ is pseudo real closed, then $F(\mathbf i)$ is pseudo algebraically closed since $F(\mathbf i)$ has no orderings.   
Hence $F$ has virtual cohomological dimension $\leq 1$, and in particular $\Gamma(F)$ is real projective by Haran's Theorem \ref{haran}. 
\end{rem}

The following result of Haran--Jarden is a strong converse of the remark above (see \cite[Theorem 10.4 on page 487]{HJ}).

\begin{thm}[Haran--Jarden]\label{haran-jarden} 
For every real projective profinite group $G$, there is a pseudo real closed field $F$ whose absolute Galois group is isomorphic to $G$. \qed
\end{thm}

\begin{rem}
Therefore, the Hopkins--Wickelgren formality conjecture for the class of fields of virtual cohomological dimension $\leq 1$ is the same as for the class of pseudo real closed fields, and it is a purely group-theoretical problem for the class of real projective profinite groups. 
\end{rem}


\begin{defn} 
Let $\Gamma$ be a profinite group. 
Let $\mathcal Y(\Gamma)$ denote the subset of elements of order $2$ in $\Gamma$. 
We equip $\mathcal Y(\Gamma)$ with the subset topology. 
Let $\mathcal X(\Gamma)$ denote the quotient of $\mathcal Y(\Gamma)$ by the conjugation action of $\Gamma$. 
We equip $\mathcal X(\Gamma)$ with the quotient topology. 
Note that when $\Gamma$ is the absolute Galois group of a field $F$ then $\mathcal X(\Gamma)$ is just the real spectrum of $F$ equipped with its Harrison topology by classical Artin--Schreier theory. 
\end{defn}

Let $C(\Xh(\Gamma),\F_2)$ be the ring of continuous functions from $\mathcal X(\Gamma)$ to $\mathbb \F_2$ where we equip the latter with the discrete topology. 

\begin{thm}[Scheiderer]\label{cor:scheiderer}
Let $\Gamma$ be a real projective group and $\mathcal X(\Gamma)$ be as above. 
Then, for each $n \ge 0$, there is a natural homomorphism 
\begin{align}\label{eq:map_pi_n}
\pi_n \colon H^n(\Gamma,\F_2) \to C(\Xh(\Gamma),\F_2)
\end{align}
which is an isomorphism for $n > 1$ and surjective for $n=1$. 
\end{thm}
\begin{proof} 
In \cite{Sch} Scheiderer constructs his theory of real \'etale cohomology in two cases: 
for schemes over which $2$ is invertible, and for profinite groups. 
In the case we are interested in, real projective groups, the two approaches are actually equivalent by Theorem \ref{haran-jarden}. 
The assertion then follows from \cite[Proposition 2.15]{Sch3}. 

We will briefly indicate how to use Theorem \ref{haran-jarden} and the results of \cite{Sch} to deduce the result. 
By Theorem \ref{haran-jarden}, we can find a field $F$ such that $\Gamma =\Gamma(F)$. 
By \cite[Theorem 6.6, page 61]{Sch} applied to the spectrum of $F$ and the constant discrete $\Gamma$-module $A=\F_2$,  there is a long exact sequence of the form 
\begin{align}\label{eq:ses_Thm6.6}
\cdots \to H^n_b(F,\F_2) \to H^n(\Gamma,\F_2) \xto{p^n_F} H^0(\Xh(\Gamma), \mathcal{H}^n(\F_2)) \to H^{n+1}_b(F,\F_2) \to \cdots 
\end{align}
where $\mathcal H^n(\F_2)$ denotes the $\Gamma$-equivariant higher direct image sheaf constructed in \cite[Section 8]{Sch} (see especially \cite[Remark 8.9, page 95]{Sch}), 
and $H_b$ denotes cohomology group with respect to the $b$-topology on $\mathrm{Spec}(F)$ as defined in \cite[Definition 2.3, page 11]{Sch}. 
For the identification of sequence \eqref{eq:ses_Thm6.6} with the long exact sequence of \cite[Theorem 6.6, page 61]{Sch}, 
we recall that $\Xh(\Gamma)$ is the real spectrum of the field $F$ (see \cite[9.2, page 97] {Sch}). 
We can then use \cite[Remarks 9.5 and 9.6, page 99]{Sch} to identify the remaining terms. 
In particular, we can identify 
the map $p^n_F$ in \eqref{eq:ses_Thm6.6} with the map $\pi_n$ in \eqref{eq:map_pi_n}.   
By \cite[Theorem 7.3, page 70]{Sch} (see also \cite[Corollary 9.8]{Sch}), 
the cohomological $2$-dimension of $F$ with respect to the $b$-topology is equal 
to the virtual cohomological $2$-dimension of $\Gamma$. 
Since the latter is $1$ by Theorem \ref{haran} and Corollary \ref{cor:vcd_one}, the group $H^{n}_b(F,\F_2)$ in sequence \eqref{eq:ses_Thm6.6} vanishes for $n>1$.  
Hence, since sequence \eqref{eq:ses_Thm6.6} is exact, the map $\pi_n$ in \eqref{eq:map_pi_n} is an isomorphism for $n>1$ and is still surjective for $n=1$. 
\end{proof}

\begin{rem}
We note that there is also a group-theoretic proof of Theorem \ref{cor:scheiderer} in \cite[Section 12]{Sch}. 
Since the virtual cohomological $2$-dimension of a real projective group is $1$, and $\Gamma$ contains no subgroup isomorphic to $\mathbb Z/2\times\mathbb Z/2$ by Haran's Theorem \ref{haran}, \cite[Theorem 12.13, page 151]{Sch} applies with coefficients the constant $\Gamma$-module $A=\F_2$, and the map $\pi_n$ in \eqref{eq:map_pi_n} is an isomorphism for every $n>1$. 
The case $n=1$ is proven in \cite[Remark 12.20.2, page 158]{Sch} where the argument of the second paragraph applies since every involution is self-centralising.  
A proof of Theorem \ref{cor:scheiderer} using Artin--Schreier structures which are discussed in \cite[Appendix B]{Sch} can also be deduced from \cite[Proposition B.8.1, page 257]{Sch}. 
\end{rem}


\section{Massey products and formality}\label{section:Massey_products}

In this section we show that formality implies strong Massey vanishing. 
This is a well known fact which can for example be deduced from the theory of $A_{\infty}$-algebras. 
However, for a lack of reference known to the authors to a proof which avoids the $A_{\infty}$-structures, 
especially in positive characteristic, 
and for the convenience of the reader we provide direct and self-contained proofs for the assertions needed. 
In this section we only assume familiarity with some basic terminology for model categories, 
and we recall and use a basic fact about weak equivalences in Lemma \ref{lemma:wetozigzag}.

\begin{notn} Let $\F$ be a field, and let $\DGA$ denote the category of associative differential graded algebras over $\F$. 
We will often refer to the objects in $\DGA$ just as dg-algebras. 
We will often denote the multiplication on objects in $\DGA$ as a cup product $\cup$. 
Let $C^*$ be a dg-algebra with differential $\delta \colon C^*\to  C^{*+1}$, and cohomology $H^*=\textrm{Ker}(\delta)/\textrm{Im}(\delta)$. 
We will often write $C$ instead of $C^*$. 
\end{notn}

We consider $\DGA$ as a model category with the standard model structure, as constructed in \cite[Theorem 5, page 58]{jardine}. 
It is determined by the following classes of weak equivalences and fibrations: a map is a  \emph{weak equivalence} if its underlying map of chain complexes is a quasi-isomorphism; and it is a \emph{fibration} if it is degree-wise surjective. 
Recall that maps which are both weak equivalences and fibrations are also called \emph{acyclic fibrations}. 
The \emph{cofibrations} are the maps which have the left lifting property with respect to all acyclic fibrations.

\begin{defn} 
Let $C$ be a dg-algebra with cohomology $H$. 
We consider $H$ as an object in $\DGA$ with trivial differential.  
Then $C$ is called {\it formal} if there is a diagram 
\begin{align}\label{eq:qisoformaldefn}
C \leftarrow T \to H
\end{align}
of weak equivalences of dg-algebras. 
\end{defn}

\begin{rem}\label{rem:zigzag_resolution}
We recall that, since all objects are fibrant in the standard model structure on $\DGA$, a sequence of quasi-isomorphisms 
\begin{align*}
C \leftarrow T_1 \to T_2 \leftarrow T_3 \to \ldots \leftarrow T_n \to B 
\end{align*} 
in $\DGA$ can be reduced to a sequence $C \leftarrow T \to B$ of just two quasi-isomorphism. 
\end{rem}


\begin{lemma}
\label{lemma:wetozigzag}
Let $f \colon A\to B$ be a weak equivalence of dg-algebras. 
Then there is a commutative diagram of morphisms of dg-algebras
\begin{align}
\label{diagram:wetozigzag}
\xymatrix{ \widetilde A\ar[r]^{j} \ar[d]_{p}& \widetilde B\ar[d]^{q} \\
	A\ar[r]^{f}  & B}
\end{align}
such that all maps are weak equivalences, the map $j$ is a homotopy equivalence, and $p$ and $q$ are fibrations. 
\end{lemma}

\begin{proof}
This is a formal consequence of the model structure on $\mathbf{DGA}$. 
Let 
\[
\xymatrix{A \ar[r]^-{i_1} & B_1 \ar[r]^-{q_1} & B}
\]
be a factorization of $f$ into a cofibration $i_1$ followed by an acyclic fibration $q_1$. Since $f$ is a weak equivalence, $i_1$ is a weak equivalence, too. 
If $B_1$ is not cofibrant, we choose a cofibrant replacement $\widetilde{B} \xto{\widetilde{q}} B_1$ such that $\widetilde{q}$ is an acyclic fibration and $\widetilde{B}$ is cofibrant. 
The composition $q=\widetilde{q}\circ q_1$ is an acyclic fibration as well. 
If $A$ is not cofibrant, we choose again a cofibrant replacement $\widetilde{A} \xto{p} A$ such that $p$ is an acyclic fibration and $\widetilde{A}$ is cofibrant. 
Now we consider the solid diagram 
\begin{align*}
\xymatrix{
& \widetilde{B} \ar[d]^{q} \\
\widetilde{A} \ar@{.>}[ur]^j \ar[r]_{f \circ p} & B.
}
\end{align*}
Since $\widetilde{A}$ is cofibrant and $q$ is an acyclic fibration, there is a dotted lift $j$ which makes the diagram commute. As $f \circ p$ is also a weak equivalence, $j$ must be a weak equivalence. Since $\widetilde{A}$ and $\widetilde{B}$ are cofibrant and fibrant objects (every object in $\DGA$ is fibrant), $j$ is a homotopy equivalence, in the sense that there is a map $\widetilde{B}\xto{\widetilde{h}} \widetilde{A}$ which is an inverse to $j$ up to homotopy. 
This proves the claim. 
\end{proof}

The existence of diagram \eqref{diagram:wetozigzag} will allow us to lift Massey products also along quasi-isomorphisms of dg-algebras. 

\begin{prop} 
\label{prop:liftofMasseyproducts}
Let $f \colon A\to B$ be a weak equivalence of dg-algebras, and 
let $a_1,a_2,\ldots,a_n$ be classes in $H^*(A)$.  
Let $b_1,b_2,\ldots,b_n$ be their images in $H^*(B)$. 
Then $f_* \colon H^*(A) \to H^*(B)$ maps the Massey product set $\langle a_1,a_2,\ldots,a_n \rangle$ bijectively onto the Massey product set $\langle b_1,b_2,\ldots,b_n \rangle$.
\end{prop}
\begin{proof}
Let $\{a_{i,i+j}\} \subset A$ for $1\leq i<i+j\leq n+1$ be a defining system for the $n$-fold Massey product of $a_1,\ldots,a_n$. 
Since $f$ commutes with the differentials and cup-products, it sends the defining system $\{a_{i,i+j}\}$ to a defining system $\{b_{i,i+j} := f(a_{i,i+j})\} \subset B$ for the $n$-fold Massey product of $b_1,\ldots,b_n$. 
Hence $f_*$ sends elements in the Massey product set $\langle a_1,a_2,\ldots,a_n \rangle$ to elements in the Massey product set $\langle b_1,b_2,\ldots,b_n \rangle$. 
More concretely, $f_*$ sends the class of the cocycle $\sum_{k=2}^n \bar{a}_{1,k}\cup a_{k,n+1}$ in $\langle a_1,\ldots,a_n \rangle$ 
to the class of the cocycle $\sum_{k=2}^n \bar{b}_{1,k}\cup b_{k,n+1}$ in $\langle b_1,\ldots,b_n \rangle$. 
Since $f_*$ is an isomorphism and therefore injective, it remains to show that every defining system for the $n$-fold Massey product of $b_1,\ldots,b_n$ 
can be lifted to a defining system for the $n$-fold Massey product of $a_1,\ldots,a_n$.

By Lemma \ref{lemma:wetozigzag}, in order to prove the claim for $f$, it suffices to prove the corresponding claims for the maps $p$, $q$ and $j$ as in diagram \eqref{diagram:wetozigzag}.  
The claim for $j$ follows from the following lemma.  

\begin{lemma} 
\label{lem:liftofMasseyproducts_htpyequiv}
If $f$ is a homotopy equivalence of dg-algebras, then $f_*$ restricts to a bijection between $\langle a_1,a_2,\ldots,a_n \rangle$ and $\langle b_1,b_2,\ldots,b_n \rangle$. 
\end{lemma}
\begin{proof} 
Since $f$ is a homotopy equivalence, we can choose a morphism of dg-algebras $g \colon B \to A$ of $f$ such that there are quasi-isomorphisms $g \circ f \simeq \id_A$ and $f \circ g \simeq \id_B$.  
In particular, $g$ is itself a quasi-isomorphism of dg-algebras, 
and $g$ sends the defining system $\{b_{i,i+j}\}$ for the $n$-fold Massey product of $b_1,\ldots,b_n$ to a defining system $\{a'_{i,i+j}:= g(b_{i,i+j})\} \subset A$ for the $n$-fold Massey product of $a_1,\ldots,a_n$. 
The defining system $\{a'_{i,i+j}:= g(b_{i,i+j})\}$ is a lift of the defining system $\{b_{i,i+j}\}$. 
In fact, since $\sum_{k=2}^n \bar{b}_{1,k}\cup b_{k,n+1}$ is a cocycle in $A$, the quasi-isomorphism $g \circ f \simeq \id_A$ implies that 
\[
\sum_{k=2}^n \bar{a}_{1,k}\cup a_{k,n+1} - \sum_{k=2}^n \bar{a}'_{1,k}\cup a'_{k,n+1}
\]
is a coboundary in $A$. 
Thus, the cohomology classes of the cocycles $\sum_{k=2}^n \bar{a}_{1,k}\cup a_{k,n+1}$ and $\sum_{k=2}^n \bar{a}'_{1,k}\cup a'_{k,n+1}$, respectively, are the same element in the Massey product set $\langle a_1,a_2,\ldots,a_n \rangle \subset H^*(A)$. 
\end{proof}

To prove the claim of Proposition \ref{prop:liftofMasseyproducts} for the acyclic fibrations $p$ and $q$, we are going to use the following basic statements about quasi-isomorphisms. 

\begin{lemma}\label{lemma:a1} 
Let $f \colon A\to B$ be an acyclic fibration of dg-algebras.
Let $b \in B^n$ be such that $\delta b=0$. 
Then there is an $a\in A^n$ such that $f(a)=b$ and $\delta a=0$.
\end{lemma}

\begin{proof} 
Since $f$ is a quasi-isomorphism, there is an $a'\in A^n$ such that $\delta a'=0$ and $f(a')$ and $b$ have the same cohomology class. 
In other words, there is a $c\in B^{n-1}$ such that $f(a')=b-\delta c$. Since $f$ is surjective, there is a $d\in A^{n-1}$ such that $f(d)=c$. Set $a=a'+\delta d$. Then $\delta a=\delta a'+\delta^2 d=0$, and 
\[
f(a)=f(a')+f(\delta d)=f(a')+\delta f(d)=b-\delta c+\delta c=b. \qedhere
\]
\end{proof}

\begin{lemma}\label{lemma:a2} 
Let $f \colon A\to B$ be an acyclic fibration of dg-algebras. 
Let $d\in A^n$ and $c \in B^{n-1}$ be such that $\delta d=0$ and $\delta c=f(d)$. Then there is a $\bar{d}\in A^{n-1}$ such that $f(\bar{d})=c$ and $\delta \bar{d}=d$.
\end{lemma}

\begin{proof} Since $f$ is a quasi-isomorphism and the cohomology class of $f(d)$ is trivial, there is a $d'\in A^{n-1}$ such that $\delta d'=d$. Then $\delta c=f(\delta d')=\delta f(d')$, so $\delta(c-f(d'))=0$. By Lemma \ref{lemma:a1} above there is an $e\in A^{n-1}$ such that $\delta e=0$ and $f(e)=c-f(d')$. Set $\bar{d}=d'+e$. Then $\delta(\bar{d})=\delta(d')+\delta(e)=d+0=d$, and $f(\bar{d})=f(d')+f(e)=f(d')+c-f(d')=c$.
\end{proof}


We are now ready to conclude the proof of Proposition \ref{prop:liftofMasseyproducts}. 
Since $p$ and $q$ are acyclic fibrations, the proof is reduced to the case when $f$ is an acyclic fibration. 
Therefore we assume from now on that $f$ is an acyclic fibration. 
We need to show 
that any defining system for the $n$-fold Massey product of $b_1,\ldots, b_n$ can be lifted to a defining system of $a_1,\ldots,a_n$. 
Let $\{b_{i,i+j}\} \subset B$ for $1\leq i<i+j\leq n+1$ be a defining system for the $n$-fold Massey product of $b_1,\ldots,b_n$. 
By Lemma \ref{lemma:a1}, we can choose cocycles $a_{1,2},a_{2,3},\ldots,a_{n,n+1} \in A$ with $f(a_{i,i+1})=b_{i,i+1}$ and $a_{i,i+1}$ represents the cohomology class $a_i$ for each $i$. 
Next, for a given $j \geq 2$, assume that we have constructed elements $a_{i,i+l}$ for all $i$ and all $l<j$ satisfying 
$f(a_{i,i+l}) = b_{i,i+l}$ 
and
\begin{align}\label{eq:ailMassey}
\delta a_{i,i+l} & = \sum_{k=i+1}^{i+l-1} \bar{a}_{i,k}\cup a_{k,i+l}
\end{align}
where we recall that $d_{i,k}$ denotes the degree of $a_{i,k}$ and $\bar{a}_{i,k}=(-1)^{1+d_{i,k}}a_{i,k}$. 
We claim that $d:=\sum_{k=i+1}^{i+j-1} \bar{a}_{i,k} \cup a_{k,i+j}$ is a cocycle. 
Applying $\delta$ and using $ (-1)^{1+d_{i,k}} \cdot  (-1)^{d_{i,k}}=-1$ yields 
\begin{align*}
\delta(d) = 
\sum_{k=i+1}^{i+j-1} \left( (-1)^{1+d_{i,k}}\delta(a_{i,k})\cup a_{k,i+j}  -a_{i,k} \cup \delta(a_{k,i+j}) \right).  
\end{align*}
Using $\delta(a_{i,i+1}) = 0  =  \delta(a_{i+j-1,i+j})$ and the relations given by \eqref{eq:ailMassey} we get 
\begin{align*}
\delta (d)  = &  \sum_{k=i+2}^{i+j-1}  (-1)^{1+d_{i,k}} \left( \sum_{s=i+1}^{k-1}  (-1)^{1+d_{i,s}}a_{i,s} \cup a_{s,k} \right)  \cup a_{k,i+j}  \\
 & - \sum_{k=i+1}^{i+j-2} a_{i,k}\cup \left( \sum_{r=k+1}^{i+j-1}  (-1)^{1+d_{kr}} a_{k,r}\cup a_{r,i+j} \right).  
\end{align*}
Now we use that, for each triple $i<s<k$, we have $1+d_{ik}=d_{is}+d_{sk}$ and get 
\begin{align*}
& \delta (d)  =   \sum_{k=i+2}^{i+j-1} \left( \sum_{s=i+1}^{k-1}  a_{i,s} \cup \bar{a}_{s,k} \right)  \cup a_{k,i+j}  
 - \sum_{k=i+1}^{i+j-2} a_{i,k}\cup \left( \sum_{r=k+1}^{i+j-1}  \bar{a}_{k,r}\cup a_{r,i+j} \right)\\
 = &  \left( \sum_{i+1 \le s < k \le i+j-1}  a_{i,s} \cup \bar{a}_{s,k}  \cup a_{k,i+j} \right) 
 - \left( \sum_{i+1\le k < r \le i+j-1} a_{i,k}\cup \bar{a}_{k,r}\cup a_{r,i+j} \right) = 0. 
\end{align*}
Hence we can apply Lemma \ref{lemma:a2} with $d=\sum_{k=i+1}^{i+j-1} (-1)^{1+d_{i,k}} a_{i,k} \cup a_{k,i+j}$ and $c=b_{i,i+j}$ to get an element $a_{i,i+j} \in A$ 
which satisfies 
\begin{align*}
f(a_{i,i+j}) = b_{i,i+j} ~ \text{and}~
\delta a_{i,i+j} = \sum_{k=i+1}^{i+j-1}\bar{a}_{i,k}\cup a_{k,i+j}. 
\end{align*}
Continuing this process we get a defining system for the classes $a_1,a_2,\ldots,a_n$. 
\end{proof}

As a consequence we get that formality implies strong Massey vanishing: 

\begin{thm}\label{thm:formalityandMasseyproducts}
Every formal differential graded algebra satisfies strong Massey vanishing.  
\end{thm}
\begin{proof} Let $C$ be a formal dg-algebra and let $H$ denote the cohomology of $C$. 
By Proposition \ref{prop:liftofMasseyproducts}, it suffices to show that if we consider $H$ as a dg-algebra with trivial differential, then for every $n$-tuple $a_1,\ldots,a_n$ in $H$ of elements whose neighbouring cup products vanish the corresponding Massey product is defined and vanishes in $H$. 
Indeed, let $\{a_{i,i+j}\} \in H$ for $1\leq i<j\leq n+1$ and $(i,j)\neq(1,n+1)$ be such that $a_{i,i+1}=a_i$ for every index $i$ and $a_{i,j}=0$, if $|i-j|>1$.
Then $\{a_{i,i+j}\}$ is a defining system for the Massey product of $a_1,\ldots,a_n$ whose corresponding Massey product vanishes. 
\end{proof}


\section{Graded Hochschild cohomology and formality}\label{section:HH}

One way to prove the formality of a dg-algebra is to show that certain Hochschild cohomology groups vanish. 
We are first going to recall the definition of graded Hochschild cohomology and then discuss the relation to formality. 
Since we only consider algebras and dg-algebras over the field $\F_2$, we will from now on work over $\F_2$. 
Unless stated otherwise $\otimes$ will denote $\otimes_{\F_2}$.

\begin{defn} 
Let $A=\bigoplus_{i \ge 0}A_i$ be a graded $\F_2$-algebra. 
A {\it graded $A$-module} is an $A$-module $M$ with a decomposition $M=\bigoplus_{i\in\mathbb Z}M_i$ into subgroups such that $A_i\cdot M_j\subseteq M_{i+j}$ for every $i \ge 0, j\in\mathbb Z$. 
A homomorphism of graded $A$-modules from $M=\bigoplus_{i\in\mathbb Z}M_i$ to $N=\bigoplus_{i\in\mathbb Z}N_i$ is an $A$-module homomorphism $\phi \colon M \to N$ which respects the gradings, i.e., we have $\phi(M_i)\subseteq N_i$ for every $i$. 
Let $\underline{\mathrm{Hom}}_A(M,N)$ denote the set of graded $A$-module homomorphisms from $M$ to $N$; it has the structure of an $A$-module. 
Graded $A$-modules with graded $A$-module homomorphisms form an abelian category.  
For $s\in \Z$ and a graded $A$-module $M$, we write $M[s]$ for the graded $A$-module given in degree $n$ by $M[s]_n = M_{n+s}$.  
\end{defn}

Let $A^e = A\otimes A^{\op}$ denote the enveloping algebra of $A$ where $A^{\op}$ denotes the opposite algebra of $A$ (see e.g., \cite[Section 1.1]{witherspoon}). 
The ring $A^e$ inherits a structure of a graded $\F_2$-algebra from $A$. 
We now define the graded Hochschild cohomology groups of a graded $\F_2$-algebra $A$ with coefficients in a graded $A^e$-module $M$. 
First, we recall the bar complex (see e.g., \cite[Section 1.1]{witherspoon}). 
For $n \ge 0$, we write $B_n(A) = A^{\otimes (n+2)}$ and define the map $(d_B)_n=d_n \colon B_{n}(A) \to B_{n-1}(A)$ by 
\[
d_n(a_0 \otimes a_1 \otimes \cdots \otimes a_{n+1}) = \sum_{i = 0}^n a_0 \otimes \cdots \otimes a_ia_{i+1} \otimes \cdots \otimes a_{n+1}
\]
for $a_0, \ldots, a_{n+1} \in A$. 
Then the \emph{bar complex} of $A$ is the complex 
\[
(B(A),d_B): \quad \cdots \xto{d_3} A^{\otimes 4} \xto{d_2} A^{\otimes 3} \xto{d_1} A^{\otimes 2} \xto{\mu} A \to 0
\] 
where $\mu$ denotes the multiplication on $A$. 
For every $n\ge 0$, $B_n(A)$ may be considered as an $A^e$-module via 
\[
(a \otimes a') \cdot (a_0 \otimes \cdots \otimes a_{n+1}) = (aa_0) \otimes a_1 \otimes \cdots \otimes a_n \otimes (a_{n+1}a'). 
\] 
We equip $B_n(A)$ with the grading induced by the grading of $A$, and we consider $d_n$ as a graded homomorphism.

\begin{defn}\label{def:HHgraded}
Let $M$ be a graded $A^e$-module. 
The graded Hochschild cohomology group $\HH^{n,s}(A,M)$ is defined as the $n$th cohomology 
of the cochain complex $(\uHom_{A^e}(B_*(A),M[s]),d_B^*)$ with differential $d_B^*f = f \circ d_B$. 
\end{defn}

\begin{rem}\label{rem:computing_graded_HH}
We now recall 
that graded Hochschild cohomology can be computed using a slightly simpler complex. 
For every $n$, there is a natural $\F_2$-linear isomorphism 
\begin{align}\label{eq:contractingiso_bar}
\uHom_{A^e}(B_n(A), M) \xto{\cong} \uHom_{\F_2}(A^{\otimes n}, M) 
\end{align}   
defined by sending a map $f$ to the $\F_2$-linear map which sends $a_1 \otimes \cdots \otimes a_n$ to $f(1\otimes a_1 \otimes \cdots \otimes a_n \otimes 1)$. 
The maps $d^*_{n} \colon \uHom_{\F_2}(A^{\otimes n}, M) \to \uHom_{\F_2}(A^{\otimes (n+1)}, M)$, 
defined for every $n$ by 
\begin{align*}
d^*_n(f)(a_1 \otimes \cdots \otimes a_{n+1}) & = a_1f(a_2 \otimes \cdots \otimes a_{n+1}) \\
& + \sum_{i=1}^n f(a_1 \otimes \cdots \otimes a_ia_{i+1} \otimes \cdots \otimes a_{n+1}) \\
& + f(a_1 \otimes \cdots \otimes a_n)a_{n+1}, 
\end{align*}
turn $(\uHom_{\F_2}(A^{\otimes *}, M),d^*)$ into a cochain complex. 
We can then check that isomorphism \eqref{eq:contractingiso_bar} induces an isomorphism of cochain complexes 
\begin{align*}
(\uHom_{A^e}(B_*(A), M),d_B^*) \xto{\cong} (\uHom_{\F_2}(A^{\otimes *}, M), d^*). 
\end{align*}   
Hence, for every $n$ and $s$, we have an isomorphism 
\begin{align}\label{eq:HH_iso_graded_ST}
\HH^{n,s}(A,M) \cong H^n(\uHom_{\F_2}(A^{\otimes *}, M[s]), d^*).
\end{align} 
In fact, isomorphism \eqref{eq:HH_iso_graded_ST} is used to define Hochschild cohomology in \cite[Section 4, page 85]{Seidel-Thomas} (see also \cite[Section 4]{BKS}).  
\end{rem}


We will now recall a criterion on the Hochschild cohomology of a graded algebra which implies formality. 
In fact, it implies the following stronger property.

\begin{defn}\label{def:intrinsic_formality} 
A graded $\F_2$-algebra $A$ is called \emph{intrinsically formal} 
if every differential graded algebra $C$ with $H^*(C)\cong A$ is formal. 
\end{defn}

\begin{rem}
We note that intrinsic formality is a property formulated for graded algebras while formality is a property formulated for dg-algebras. 
However, one may call a dg-algebra $C$ intrinsically formal if its cohomology algebra is. 
Intrinsic formality is a much stronger property than formality. 
\end{rem} 

An equivalent characterisation of intrinsic formality is given as follows: 

\begin{lemma}\label{lem:intrinsic_formality_alternative}
Let $A$ be a graded $\F_2$-algebra. 
Then $A$ is intrinsically formal if and only if any two differential graded algebras with cohomology isomorphic to $A$ are quasi-isomorphic. 
\end{lemma}
\begin{proof}
Let $C$ be a dg-algebra with $H^*(C) \cong A$. 
If any two dg-algebras with cohomology isomorphic to $A$ are quasi-isomorphic, then $C$ and $H^*(C)$ are quasi-isomorphic. 
Hence $C$ is formal. 
This shows that $A$ is intrinsically formal. 

To prove the other implication,  
let $C_1$ and $C_2$ be dg-algebras with cohomology algebras $H^*(C_1) \cong A \cong H^*(C_2)$. 
If $A$ is intrinsically formal, then $C_1$ and $C_2$ are formal. 
This means that there are quasi-isomorphisms of dg-algebras $C_1 \leftarrow T_1 \to A$ and $A \leftarrow T_2 \to C_2$ where we consider $A$ as a dg-algebra with trivial differential. 
By Remark \ref{rem:zigzag_resolution}, we can then find a dg-algebra $T$ and quasi-isomorphisms of dg-algebras $C_1 \leftarrow T \to C_2$. 
Thus, $C_1$ and $C_2$ are quasi-isomorphic. 
\end{proof}

The following theorem is a special case of a result due to Kadeishvili (see \cite{Kadeishvili82} and \cite{Kadeishvili88}). 
For a proof we refer to the work of Seidel and Thomas \cite[Theorem 4.7, page 85]{Seidel-Thomas}. 

\begin{thm}[Kadeishvili]
\label{thm:KadeishvilicritST}
Let $A=\bigoplus_{i\ge 0}A_i$ be a non-negatively graded $\F_2$-algebra with $A_0=\F_2$. 
Assume that 
\[
\HH^{n,2-n}(A)=0 ~\text{for all}~n \ge 3.
\] 
Then $A$ is intrinsically formal. \qed 
\end{thm}

Theorem \ref{thm:KadeishvilicritST} and Lemma \ref{lem:intrinsic_formality_alternative} then imply: 
 
\begin{cor}\label{cor:Kadeishvilicrit}
Let $C$ be a dg-algebra over $\F_2$ with $C^0=\F_2$.  
Assume that 
\[
\HH^{n,2-n}(H^*(C)) = 0 ~\text{for all}~n \ge 3.
\] 
Then $C$ is formal. \qed
\end{cor}

We will refer to the above criteria using the following terminology:

\begin{defn}\label{def:Kadeishvilisvanishing}
For a non-negatively graded $\F_2$-algebra $A=\bigoplus_{i\ge 0}A_i$ with $A_0=\F_2$, we say that $A$ satisfies {\it Kadeishvili's vanishing} if 
\[
\HH^{n,2-n}(A)=0 \textrm{ for every $n\geq 3$.} 
\] 
For a dg-algebra $C$ over $\F_2$ with $C^0=\F_2$, we say that $C$ satisfies {\it Kadeishvili's vanishing} if the graded $\F_2$-algebra $H^*(C)$ satisfies Kadeishvili's vanishing. 
\end{defn}

\begin{rem}
For the reader familiar with $A_{\infty}$-algebras we note that the condition of Theorem \ref{thm:KadeishvilicritST} can be interpreted and the proof can be summarised as follows: 
Let $A$ be a graded algebra and let $C$ be a dg-algebra with an isomorphism of graded algebras $A \xto{\cong} H^*(C)$. 
Recall that we can view $C$ as as $A_{\infty}$-algebra with trivial higher multiplication maps $m_n$ for $n\ge 3$. 
Kadeishvili showed that it is always possible to equip $A$ with the structure of an $A_{\infty}$-algebra such that there is a quasi-isomorphism of $A_{\infty}$-algebras $q \colon A \to C$ which lifts the isomorphism $A \xto{\cong} H^*(C)$ of graded algebras (see \cite{Kadeishvili82} and \cite{Kadeishvili88}).  
Since the higher level $A_{\infty}$-maps $m_n$ on $A$ may be nontrivial for $n\ge 3$, the $A_{\infty}$-algebra structure on $A$ may be different from the one it inherits as a graded algebra. 
Moreover, $q$ is not a morphism of dg-algebras in general. 
However, each map $m_n$ is a cocycle in the Hochschild complex and represents a cohomology class in $\HH^{n,2-n}(A)$. 
Thus, if $\HH^{n,2-n}(A) = 0$ for all $n\ge 3$, then all higher multiplication maps on $A$ must be trivial. 
One can then lift the $A_{\infty}$-morphism $q \colon A \to C$ via a certain tensor algebra $T$ to a zigzag of quasi-isomorphisms $A \leftarrow T \to C$ of dg-algebras. 
Then it suffices to compose with the algebra isomorphism $H^*(C) \cong A$ to conclude that $C$ is formal. 
We refer to \cite[Proof of  Theorem 4.7]{Seidel-Thomas} for the details.
\end{rem}


\section{Graded Hochschild cohomology of Koszul algebras}\label{section:Koszul}

We continue to work over the field $\F_2$. 
Unless stated otherwise $\otimes$ will denote $\otimes_{\F_2}$. 

\begin{defn}
A \emph{quadratic algebra} is a non-negatively graded $\F_2$-algebra $A=\bigoplus_{i\ge 0}A_i$ such that $A_0=\F_2$ and $A$ is generated over $\F_2$ by $A_1$ with relations of degree two. 
Explicitly, let 
\[
T(A_1) = \F_2 \oplus A_1 \oplus (A_1\otimes A_1) \oplus \cdots = \bigoplus_{i\ge 0} A_1^{\otimes i}
\]
be the free tensor algebra of the $\F_2$-vector space $A_1$. 
Then $A$ is quadratic if the canonical map $\tau \colon T(A_1) \to A$ is surjective and $\ker(\tau)$ is generated by its component $R = \ker(\tau)\cap (A_1\otimes A_1)$ of degree two, so that $A=T(A_1)/(R)$, where $(R)$ denotes the ideal generated by $R$.  
A quadratic algebra $A$ is called \emph{locally finite} if each $A_i$ is a finite-dimensional vector space over $\F$. 
\end{defn} 


For a quadratic algebra $A=\bigoplus_{i\ge 0}A_i$ we denote by $A_+$ the ideal $A_+=\bigoplus_{i >0}A_i$. 
We will later use the following natural way to create new quadratic algebras out of old (see for e.g. \cite[Chapter 3.1]{ppqa}):

\begin{defn}\label{def:connected_sum}
Let $A=\bigoplus_{i\ge 0}A_i$ and $B=\bigoplus_{i\ge 0}B_i$ be quadratic algebras. 
The \emph{connected sum} of $A$ and $B$, denoted by $A\uoplus B$, is the graded ring with $(A\uoplus B)_0=\F_2$, $(A\uoplus B)_i = A_i \oplus B_i$ for $i>0$ and multiplication $A_+B_+$ and $B_+A_+$ is set to be zero. 
\end{defn}

The connected sum $C:=A \uoplus B$ is a quadratic algebra. 
Let $R_A\subseteq A_1\otimes A_1$ and $R_B \subseteq B_1 \otimes B_1$ denote the relations in the tensor algebras defining $A$ and $B$, respectively. 
Then we can write $C$ as a quotient of a tensor algebra 
\[
C = T(A_1 \oplus B_1)/(R_C)
\]
where $R_C \subseteq C_1^{\otimes 2}$ is given by $R_C = R_A \oplus R_B \oplus (A_1\otimes B_1) \oplus (B_1 \otimes A_1)$. 

To every quadratic algebra we can associate the following chain complex: 

\begin{defn}\label{def:Koszulcomplex}
Let $V$ be an $\F_2$-vector space and let $T(V)= \bigoplus_{i\ge 0} V^{\otimes i}$ be its tensor algebra. 
We consider $T(V)$ as a graded $\F_2$-algebra with $|v|=1$ for all $v \in V$. 
Let $R \subset V \otimes V$ and let $A=T(V)/(R)$ be the associated quadratic algebra where $(R)$ denotes the ideal generated by $R$ in $T(V)$. 
We denote by $K_i^i(A)$ the $\F_2$-linear subspace defined by $K_0^0(A) = \F_2$, $K_1^1(A) = V$, $K_2^2(A) = R$ and 
\begin{align*}
K_i^i(A) = \bigcap_{0\le j \le i-2} V^{\otimes j} \otimes R \otimes V^{\otimes i-j-2} \subset V^{\otimes i} ~ \text{for}~i\ge 3.
\end{align*}
We set $K_i(A)=A\otimes K_i^i(A)\otimes A$. 
We consider $K_i(A)$ as an $A^e$-module with $A$ acting on the leftmost factor $A$ and $A^{\op}$ acting on the rightmost factor $A$, respectively. 
For each $i\ge 0$, we define a homomorphism $d=d_i \colon K_i(A) \to K_{i-1}(A)$ by
\[
d_i (a\otimes x_1 \otimes \cdots \otimes x_i \otimes a') = (ax_1) \otimes x_2 \otimes \cdots \otimes x_i \otimes a' + 
a \otimes x_1 \otimes x_2 \otimes \cdots \otimes (x_ia'). 
\]
Since $R\subset V^{\otimes 2}$ generates the relations in $A$, it is clear that $d^2=0$. 
We refer to the chain complex $(K(A),d)$ as the \emph{Koszul complex of $A$}.  
\end{defn}

Let $(B(A),d_B)$ denote the bar complex of $A$.  
Since $R$ describes the relations in $A$, the natural inclusion $K_i(A) \into B_i(A)$ defines a morphism of complexes $(K(A),d) \to (B(A),d_B)$. 
The multiplication map 
$\mu \colon A\otimes A \to A$ defines morphisms of complexes from the complexes $(K(A),d)$ and $(B(A),d_B)$ to the complex $A$ concentrated in degree $0$ with trivial differential. 
However, while $\mu \colon B(A) \to A$ is always a quasi-isomorphism, $\mu \colon K(A) \to A$ may not be a quasi-isomorphism for an arbitrary quadratic algebra $A$. 
This leads to the following definition. 

\begin{defn}\label{def:Koszulity}
A quadratic algebra $A$ is called \emph{Koszul} if the morphism of complexes induced by multiplication $\mu \colon K(A) \to A$ is a quasi-isomorphism. 
\end{defn}

\begin{rem}\label{rem:Koszul_equivalent_defs}
We note that there are many equivalent definitions of the notion of a Koszul algebra. 
By \cite[Proposition 3.1]{vandenbergh}(where the notation $K'(A)$ is used for $K(A)$), 
Definition \ref{def:Koszulity} is equivalent to the ones in \cite{ppqa}, \cite{PoGalois} and \cite{Po} (see also \cite[Section 2.1]{bls}). 
\end{rem}

By \cite[Proposition 3.1]{vandenbergh}, the Koszul complex $(K(A),d)$ is a complex of free $A^e$-modules. 
Hence, if $A$ is a Koszul algebra, we can use $K(A)$ to compute Hochschild cohomology. 
In fact, since $K(A) \into B(A)$ is a quasi-isomorphism when $A$ is a Koszul algebra, 
we get the following result:

\begin{prop}\label{prop:HHviaKoszulcomplex}
Let $A$ be a Koszul algebra and $M$ be a graded $A$-bimodule. 
The canonical inclusion of chain complexes $K(A) \into B(A)$ induces an isomorphism between 
the graded Hochschild cohomology of $A$ with coefficients in $M$ and the cohomology of the cochain complex  $(\uHom_{A^{e}}(K(A),M),d^*)$ 
where $d^*$ denotes the differential induced by the differential $d$ of the chcain complex $K(A)$. \qed
\end{prop}


We now show that for a Koszul algebra an even simpler complex can be used to compute its Hochschild cohomology. 

\begin{defn}\label{def:simple_Koszul_cochain_complex}
Let $A$ be a Koszul algebra and $M$ be a graded $A$-bimodule. 
For every $n \ge 0$ and $s \in \Z$, 
let $\partial$ denote the $\F_2$-linear map 
\begin{align*}
\partial_n \colon \Hom_{\F_2}(K_n^n(A), M_{n+s}) \to \Hom_{\F_2}(K_{n+1}^{n+1}(A), M_{n+1+s})
\end{align*}
defined by setting 
\begin{align*}
\partial_n(f)(x_1\otimes \ldots \otimes x_{n+1}) =  x_1f(x_2\otimes  \ldots \otimes x_{n+1}) 
+  f(x_1 \otimes \ldots  \otimes x_{n})x_{n+1}. 
\end{align*}
It is easily verified that $\partial \circ \partial =0$. 
We refer to the cochain complex 
\begin{align}\label{eq:simple_Koszul_cochain_complex}
(\uHom_{\F_2}(K^*_*(A), M[s]), \partial)
\end{align}
as the \emph{simplified Koszul cochain complex}. 
\end{defn}

Let $A$ be a Koszul algebra and $M$ be a graded $A$-bimodule. 
For every $n$ and every integer $s$, isomorphism \eqref{eq:contractingiso_bar} descends to an isomorphism 
\begin{align}\label{eq:contractingiso}
\uHom_{A^e}(K_n(A), M[s])  \xto{\cong} \uHom_{\F_2}(K_n^n(A), M_{n+s}).  
\end{align}  
It is straight-forward to check that isomorphism \eqref{eq:contractingiso} yields an isomorphism of cochain complexes 
\begin{align}\label{eq:contractingiso_complex}
(\uHom_{A^e}(K_*(A), M[s]),d^*)  \xto{\cong} (\uHom_{\F_2}(K_*^*(A), M[s]),\partial).  
\end{align} 
By Proposition \ref{prop:HHviaKoszulcomplex}, 
the inclusion $K(A) \into B(A)$ induces an isomorphism 
between the Hochschild cohomology group $\HH^{n,s}(A,M)$ and the $n$th-cohomology of the left-hand side of \eqref{eq:contractingiso_complex}. 
Since \eqref{eq:contractingiso_complex} is an isomorphism of cochain complexes, 
this proves the following result:

\begin{prop}\label{prop:HHviaKoszulcomplex_simple}
Let $A$ be a Koszul algebra and $M$ be a graded $A$-bimodule. 
For every pair $(n,s)$, 
the composition of the restriction along the inclusion $K(A) \into B(A)$ and isomorphism \eqref{eq:contractingiso_complex}
defines an isomorphism 
\begin{align}\label{eq:diag_iso_of_HH}
\HH^{n,s}(A,M) \xto{\cong} H^n(\uHom_{\F_2}(K_*^*(A), M[s]),\partial)
\end{align}
from the graded Hochschild cohomology group $\HH^{n,s}(A,M)$ to the $n$th cohomology group of the cochain complex  
$(\uHom_{\F_2}(K^*_*(A), M[s]), \partial)$. \qed
\end{prop}

Theorem \ref{thm:KadeishvilicritST} and Proposition \ref{prop:HHviaKoszulcomplex_simple} 
then imply the following criterion for intrinsic formality: 

\begin{cor}
\label{cor:KadeishvilicritST_simplified}
Let $A$ be a Koszul algebra over $\F_2$. 
Assume that 
\[
H^n(\uHom_{\F_2}(K_*^*(A), A[2-n]),\partial) = 0 ~\text{for all}~n \ge 3.
\] 
Then $A$ is intrinsically formal. \qed
\end{cor}

Finally, to prove our main technical result in the next sections we will apply the following spectral sequence. 
We will provide a proof of the following proposition in appendix \ref{section:appendix}. 

\begin{prop}\label{prop:spectralsequence}
Let $A$ be a positively graded $\F_2$-algebra with $A_0=\F_2$. 
Let $\{A_{\alpha}\}_{\alpha}$ be a directed system of graded subalgebras such that $A$ is equal to the inductive limit $\varinjlim_{\alpha}A_{\alpha}$. 
Let $M$ be a graded $A^e$-module. 
For each integer $s$, there is a first quadrant spectral sequence 
\[
E_2^{p,q} = {\varprojlim_{\alpha}}^p \HH^{q,s}(A_{\alpha}, M) \Rightarrow  \HH^{p+q,s}(A,M)
\]
where ${\varprojlim^p_{\alpha}}$ denotes the $p$th right derived functor of the inverse limit functor. 
\end{prop}


\section{Boolean and dual graded algebras}\label{section:def_Boolean_and_dual}

In this section we introduce the two types of graded algebras for which we will compute the Hochschild cohomology of their connected sums 
in the next sections. 
For more details and background on Boolean rings we refer to appendix \ref{app:Boolean}. 

\begin{defn}\label{def:boolean} We say that a ring $B$ is {\it Boolean} if $x^2=x$ for every $x\in B$. 
For every Boolean ring $B$, let $B_*=\bigoplus_{n\ge 0}B_n$ denote the graded ring 
such that $B_0=\mathbb F_2$ and for every $n\ge 1$ we set $B_n=B$. 
For every pair of positive integers $m,n$ the multiplication $B_m\times B_n\to B_{m+n}$ is the multiplication in the ring $B=B_m=B_n=B_{m+n}$. 
We call $B_*$ the {\it associated Boolean graded algebra}. 
\end{defn}

For the next lemma, we recall that the polynomial ring $A=\F_2[x]$ is a quadratic algebra with $A_1= \F_2 \langle x \rangle$, the one-dimensional $\F_2$-vector space generated by $x$, and trivial relations.  
We also recall from Corollary \ref{atomic} that every finite Boolean ring has an orthogonal basis, 
i.e., there are elements $x_1,x_2,\ldots,x_n$ in $B$ which form a basis of $B$ as a vector space over $\F_2$ and $x_ix_j=0$ for $i\ne j$. 

\begin{lemma}\label{lemma:boolean}
Let $B$ be a finite Boolean ring, and let $x_1,x_2,\ldots,x_n$ be an orthogonal basis of $B$. 
Then the associated Boolean graded algebra $B_*$ is a quadratic algebra and is isomorphic to the connected sum of the $\F_2[x_j]$, i.e., 
\begin{align*}
B_* \cong \F_2[x_1] \uoplus \cdots \uoplus \F_2[x_n].
\end{align*}
Moreover, $B_*$ is a Koszul algebra.
\end{lemma}
\begin{proof}
That $B_*$ is isomorphic to the connected sum of the quadratic algebras $\F_2[x_j]$ follows immediately from Definition \ref{def:boolean}. 
By \cite[Chapter 3, Corollary 1.2 on page 58]{ppqa}, the connected sum of Koszul algebras is Koszul. 
Hence, in order to show that $B_*$ is Koszul, it suffices to show that $\F_2[x]$ is a Koszul algebra. 
The latter follows from \cite[Example on page 20]{ppqa} and Remark \ref{rem:Koszul_equivalent_defs} since $\F_2[x]$ is the symmetric algebra in one generator. 
Thus, $\F_2[x]$ is Koszul and hence so is $B_*$. 
\end{proof}

\begin{prop}\label{prop:Booleanisquadratic_only_Boolean}
Let $B$ be a Boolean ring and let $B_*$ be its associated Boolean graded algebra. 
Then $B_*$ is a Koszul algebra. 
\end{prop}
\begin{proof} 
This follows from the fact that every Boolean ring is the inductive limit of finite Boolean rings. 
Indeed, every ring is the inductive limit of its finitely generated subrings. 
Every subring of a Boolean ring is Boolean, and from Stone's representation theorem it is also clear that every finitely generated Boolean ring is finite. 
Also note that the functor $B\mapsto B_*$ of Definition \ref{def:boolean} maps injective morphisms to injective morphisms, so for every Boolean ring $B$ the graded $\mathbb F_2$-algebra $B_*$ is the inductive limit of graded $\mathbb F_2$-algebras of the form $A_*$ where $A \subseteq B$ is a finite Boolean ring. 	
By Lemma \ref{lemma:boolean}, it now suffices to observe that the property of being a Koszul algebra is preserved under passing to inductive limits in the category of graded algebras over a field.  
This was shown by Bruns and Gubeladze in \cite[Lemma 1.1, page 48]{BG}. 
\end{proof}


As a generalisation of the algebra of dual numbers $k[\varepsilon]/(\varepsilon^2)$, we refer to the following type of graded algebras as {\it dual algebras} (even though we do not consider them as the dual of another algebra). 

\begin{defn}\label{defn:dual}
Let $V$ be a vector space over $\F_2$. 
We call the non-negatively graded $\F_2$-algebra $V_*=\bigoplus_{j\ge 0}V_j$  with $V_0=\F_2$, $V_1=V$ and $V_j=0$ for $j\geq2$ the {\it associated dual algebra}. A dual algebra $V_*$ is quadratic with relations $R=V_1\otimes_{\F_2} V_1$. 
\end{defn}

\begin{lemma}\label{lemma:dualformal}
For every $\F_2$-vector space $V$, the dual algebra $V_*$ is Koszul. 
For $k+\mm\ge 2$, we have $\HH^{k,\mm}(V_*)=0$. 
\end{lemma} 
\begin{proof}
That $V_*$ is Koszul follows from the fact that the inclusion map of $K(V_*)$ into the bar complex is the identity. 
The vanishing of $\HH^{k,\mm}(V_*)$ for $k+\mm\ge 2$ follows from the fact that there are no nontrivial cochains in the Koszul cochain complex for $k+\mm \ge 2$, 
since $V_j=0$ for $j\ge 2$. 
\end{proof}

\begin{prop}\label{prop:Booleanisquadratic}
Let $V$ be an $\F_2$-vector space, let  $V_*$ be its associated dual algebra, $B$ be a Boolean ring and let $B_*$ be its associated Boolean graded algebra. 
Then $V_*\uoplus B_*$ is a Koszul algebra. 
\end{prop}
\begin{proof} 
By Lemma \ref{lemma:dualformal} we know that $V_*$ is Koszul, and by Proposition \ref{prop:Booleanisquadratic_only_Boolean} we know that $B_*$ is Koszul. 
By \cite[Chapter 3, Corollary 1.2 on page 58]{ppqa} the connected sum of Koszul algebras is Koszul. 
This proves the assertion. 
\end{proof} 


\section{Kadeishvili's vanishing for connected sums of dual and Boolean graded algebras - the finite case}\label{section:Boolean_finite}

In this section we are going to prove that connected sums of dual algebras 
and the Boolean graded algebras associated to a Boolean ring have vanishing Hochschild cohomology 
when we restrict to a finite subring.   
In Section \ref{section:Boolean_infinite} we extend the computation to the case of arbitrary Boolean subrings.  

We fix an $\F_2$-vector space $V$, and let $V_*$ be its associated dual algebra.  
We also fix a Boolean ring $B$ and a \emph{finite} Boolean subring $A\subset B$. 
Let $A_*$ and $B_*$ be the associated Boolean graded algebras. 
Our next goal is to compute the groups 
$\HH^{k,\mm}(V_*\uoplus A_*,V_* \uoplus B_*)$ for $\mm<0$.

By Propositions \ref{prop:HHviaKoszulcomplex_simple} and \ref{prop:Booleanisquadratic}, 
we can do this by using the simplified Koszul cochain complex \eqref{eq:simple_Koszul_cochain_complex} of $V_* \uoplus A_*$ of Definition \ref{def:simple_Koszul_cochain_complex}, i.e., we consider the cochain complex 
\begin{align*}
(\uHom_{\F_2}(K^*_*(V_* \uoplus A_*), (V_* \uoplus B_*)[\mm]), \partial). 
\end{align*}  
First we determine the $\F_2$-vector spaces $K^k_k(V_*\uoplus A_*)$. 
This will require some preparation. 

\begin{defn}\label{def:admissiblesequences}
Let $I$ denote a basis of $V$. 
Let $x_1,x_2,\ldots,x_n$ be the orthonormal basis of $A$, see Corollary \ref{atomic}. 
Let $J$ denote the set $\{x_1,x_2,\ldots,x_n\}$. 
A sequence $t_1,t_2,\ldots,t_k\in I\cup J$ is {\it admissible} if for every $j=1,2,\ldots,k-1$ such that $t_j,t_{j+1}\in J$ we have $t_j\neq t_{j+1}$. 
We say that an admissible sequence $t_1,t_2,\ldots,t_k\in I\cup J$ is {\it stable} if $t_1,t_k\in J$ and $t_1=t_k$, and it is {\it unstable} otherwise. 
\end{defn}


\begin{notn} For every sequence $\mathbf t=t_1,t_2,\ldots,t_k\in I\cup J$, 
let the same symbol $\mathbf t$ denote the element $t_1\otimes t_2\otimes\cdots\otimes t_k\in K^k_k(V_* \uoplus A_*)$ by slight abuse of notation. 
\end{notn}

\begin{lemma}\label{koszul_explicit_2} Write $Q_*= V_* \uoplus A_*$. The subspace $K^k_k(Q)\subset Q_1^{\otimes k}$ is spanned by the basis consisting of the vectors $\mathbf t$, where $\mathbf t$ is any admissible sequence. 
\end{lemma}

\begin{proof} 
Let $R_Q \subseteq Q_1 \otimes Q_1$ denote the relations in the tensor algebra defining the quadratic algebra $Q$. 
Note that the vectors $\mathbf t$, where $\mathbf t$ is any sequence in $I\cup J$, are linearly independent, so any subset of them will form a basis of their span. Therefore it will be sufficient to prove that for every index $j=1,2,\ldots,k-1$ the subspace:
$$Q_1^{\otimes{j-1}}\otimes R_Q\otimes Q_1^{\otimes{k-j-1}}\subset Q_1^{\otimes k}$$
is spanned by the vectors $\mathbf t$, where $\mathbf t=t_1,t_2,\ldots,t_k\in I\cup J$ satisfies the condition $t_j\neq t_{j+1}$ if $t_j,t_{j+1}\in J$. 
Recall that the relations in $A_*$ are generated by the vectors $x_i \otimes x_j$ for $i\neq j$. 
Let $T_Q\subset Q_1^{\otimes2}$ be the subspace spanned by the vectors
$x_j\otimes x_j$, where $j=1,2,\ldots,n$. Then $Q_1^{\otimes2}=R_Q\oplus T_Q$, and hence
\begin{equation}\label{0.2.1_Q}
Q_1^{\otimes k} = \left(Q_1^{\otimes{j-1}}\otimes R_Q\otimes Q_1^{\otimes{k-j-1}} \right) \oplus \left(Q_1^{\otimes{j-1}}\otimes T_Q\otimes Q_1^{\otimes{k-j-1}}\right).
\end{equation}
Now if
$$\underline t=\sum_{\mathbf t}a_{\mathbf t}\mathbf t=
\left(\sum_{\{t_j,t_{j+1}\}\nsubseteq J }
a_{\mathbf t}\mathbf t+\sum_{\substack{ t_j,t_{j+1}\in J \\ t_j\neq t_{j+1} } }
a_{\mathbf t}\mathbf t\right)+
\sum_{\substack{ t_j,t_{j+1}\in J \\ t_j=t_{j+1} } }a_{\mathbf t}\mathbf t
\quad(a_{\mathbf t}\in\mathbb F_2)$$
lies in $Q_1^{\otimes{j-1}}\otimes R_Q\otimes Q_1^{\otimes{k-j-1}}$, then the sum outside of the parentheses must be zero, since it lies in the second summand of the  decomposition in (\ref{0.2.1_Q}), while the sum within the parentheses lies in the first summand of the  decomposition in (\ref{0.2.1_Q}). 
The claim is now clear. 
\end{proof}


We are going to use the following fact about Boolean rings. 
\begin{lemma}\label{ideals} 
Let $x,y\in B$. Let $(x)$, $(y)$ and $(x,y)$ denote the ideals in $B$ generated by $x$, $y$, $x$ and $y$, respectively. 
Then the following assertions hold:
\begin{enumerate}
	\item[$(i)$] for every $z\in B$, we have $z\in (x,y)$ if and only if $(1+x+y+xy)z=0$; 
	\item[$(ii)$] if $xy=0$ then $(x,y)$ is the direct sum of $(x)$ and $(y)$.  
\end{enumerate}
In particular, for every $x,z\in B$ we have $z\in (x)$ if and only if $(1+x)z=0$. 
\end{lemma}

\begin{proof} If $(1+x+y+xy)z=0$ then $z=z(x+y+xy)\in (x,y)$. If $z\in (x,y)$ then there are $a,b\in B$ such that $z=ax+by$, so
$$(1+x+y+xy)z=(ax+by)+(ax+bxy)+(axy+by)+(axy+bxy)=0,$$
so claim $(i)$ is true. 
Now assume that $xy=0$. Note that $(x,y)=(x)+(y)$, so we only need to show that $(x)\cap(y)=0$. If $z\in(x)\cap(y)$ then there are $a,b\in B$ such that $z=ax=by$. Therefore
$$z=z^2=axby=abxy=0,$$
so claim $(ii)$ holds, too. 
\end{proof}

Now we assume that $j\ge 1$ and that $f \colon  K^k_k(V_*\uoplus A_*)\to B_j$ is a cocycle in the simplified Koszul cochain complex $(\uHom_{\F_2}(K^*_*(V_*\uoplus A_*), B_*[j-k]), \partial)$ of Definition \ref{def:simple_Koszul_cochain_complex}.  
For the sake of simple notation, we will identify $B_j$ with $B_1$ in all that follows. 
Let $p \colon V_1\oplus B_1\to B_1$ be the unique linear map which is zero on $V_1$ and the identity on $B_1$.
\begin{lemma}\label{gereon_relation} 
For every admissible sequence $\mathbf t=t_1,t_2,\ldots,t_k\in I\cup J$ we have:
$$f(\mathbf t)\in (p(t_1),p(t_k)).$$
\end{lemma}
\begin{proof} 
We define the element $x:=1+p(t_1)+p(t_k)+p(t_1)p(t_k)\in A_1 \subset B_1 = B$. 
We claim that $x$ satisfies 
$$xt_1=0=xt_k.$$
First, elements in $I$ are orthogonal to all elements and the claim is true for $t_1,t_k\in I$. 
Second, if $t_1\in J$, then $xt_1  =  t_1 + t_1^2 + t_1p(t_k) + t_1^2p(t_k) = t_1 + t_1  + t_1p(t_k) +  t_1p(t_k) = 0$. 
And we  have a  similar calculation if $t_k\in J$. This shows the claim. 

Hence $x\otimes t_1\otimes\cdots\otimes t_{k-1}\otimes t_k\in K^{k+1}_{k+1}(V_*\uoplus A_*)$. 
Since $f$ is a cocycle:
$$xf(t_1\otimes t_2\otimes\cdots\otimes t_k)+
f(x\otimes t_1\otimes\cdots\otimes t_{k-1})t_k=0.$$
Multiplying both sides of this equation by $x$ we get that 
$$0=x^2f(t_1\otimes t_2\otimes\cdots\otimes t_k)+
f(x\otimes t_1\otimes\cdots\otimes t_{k-1})xt_k=xf(t_1\otimes t_2\otimes\cdots\otimes t_k).$$
The claim now follows from part $(i)$ of Lemma \ref{ideals}. 
\end{proof}

\begin{defn} 
For every stable admissible sequence $\mathbf t$, by Lemma \ref{gereon_relation}, we have 
$$f(\mathbf t)=\alpha(\mathbf t)$$
for a unique $\alpha(\mathbf t)\in (p(t_1))$. Set $\beta(\mathbf t)=\alpha(\mathbf t)$ in this case. 
For every unstable admissible sequence $\mathbf t$, Lemma \ref{gereon_relation} implies 
$$f(\mathbf t)=\alpha(\mathbf t)+\beta(\mathbf t),\quad\alpha(\mathbf t)\in(p(t_1)),\beta(\mathbf t)\in(p(t_k)),$$
where $\alpha(\mathbf t),\beta(\mathbf t)$ are uniquely determined by part $(ii)$ of Lemma \ref{ideals}. 
For every admissible sequence $\mathbf t$ we call $\alpha(\mathbf t),\beta(\mathbf t)$ the {\it head, tail value} of $f(\mathbf t)$, respectively. 
\end{defn}


Let $A_k(I,n)$ denote the set of admissible sequences $t_1,t_2,\ldots,t_k\in I\cup J$, 
where we recall that $n$ denotes the cardinality of $J$, i.e., the $\F_2$-dimension of the finite subring $A \subset B$, 
which determines the associated Boolean graded algebra $A_*$ by Lemma \ref{lemma:boolean}. 

\begin{defn} \label{translations}
We define two maps $L \colon A_k(I,n)\to A_k(I,n)$ and $R \colon  A_k(I,n)\to A_k(I,n)$, the {\it left} and {\it right translations}, as follows. 
Let
$$L(t_1,t_2,\ldots,t_k)=\begin{cases}
t_1,t_2,\ldots,t_k & \text{, if the sequence is stable} ,\\
t_2,t_3,\ldots,t_k,t_1 & \text{, if the sequence is unstable, } \end{cases}$$
and similarly
$$R(t_1,t_2,\ldots,t_k)=\begin{cases}
t_1,t_2,\ldots,t_k & \text{, if the sequence is stable} ,\\
t_k,t_1,t_2,\ldots,t_{k-1} & \text{, if the sequence is unstable. } \end{cases}$$
\end{defn}
These two maps translate the sequence to the left or to the right, respectively, and put the term which drops out to the other end. 
It is also clear that they are well-defined, i.e., they map admissible sequences to admissible sequences. Note that $L$ and $R$ are inverses of each other, so we get a permutation, or a cyclic group action on $A_k(I,n)$. 
The fixed points of this action are exactly the stable admissible sequences.

\begin{lemma}\label{head/tail_Q} For every admissible sequence $\mathbf t=t_1,t_2,\ldots,t_k\in I\cup J$ we have:
$$\alpha(R(\mathbf t))=\beta(\mathbf t),$$
and equivalently:
$$\beta(L(\mathbf t))=\alpha(\mathbf t).$$
\end{lemma}

\begin{proof} The case of a stable admissible sequence is trivial. If $\mathbf t$ is unstable, the two identities above follow from the two relations:
\begin{align*}
t_kf(t_1\otimes t_2\otimes\cdots\otimes t_k)+
f(t_k\otimes t_1\otimes\cdots\otimes t_{k-1})t_k=0, \\
t_1f(t_2\otimes\cdots\otimes t_k\otimes t_1)+
f(t_1\otimes t_2\otimes\cdots\otimes t_k)t_1=0, 
\end{align*}
respectively. 
\end{proof}

\begin{rem} Note that the lemma above covers all relations between head and tail values for cocycles. 
This is true because all relations for cocycles are spanned by the relations:
\begin{equation}\label{0.7.1_Q}
t_1f(t_2\otimes t_3\otimes\cdots\otimes t_{k+1})+
f(t_1\otimes t_2\otimes\cdots\otimes t_k)t_{k+1}=0,
\end{equation}
where $\mathbf t=t_1,t_2,\ldots,t_{k+1}\in I\cup J$ is any admissible sequence.  
In particular there is no relation between head and tail values for admissible sequences lying in different orbits of $R$. Within every orbit of the right translation operator the head value of the right translate is equal to the tail value of the translated sequence, but otherwise we are free to choose the head and tail values within an orbit. 
\end{rem}

\begin{defn} We define two maps
$$l \colon A_k(I,n)\to A_{k-1}(I,n)\quad\textrm{and}
\quad r\colon A_k(I,n)\to A_{k-1}(I,n),$$
the {\it left and right truncations}, as follows. Let
$$l(t_1,t_2,\ldots,t_k)=t_1,t_2,\ldots,t_{k-1},$$
and similarly
$$r(t_1,t_2,\ldots,t_k)=t_2,t_3,\ldots,t_k.$$
These maps are clearly well-defined, i.e., they map admissible sequences to admissible sequences. 
We say that an orbit $O\subset A_k(I,n)$ of $R$ is {\it stable} if it consists of one element, and it is {\it unstable} if it consists of at least two elements. 
\end{defn}

\begin{lemma}\label{truncation} Let $O\subset A_k(I,n)$ be an unstable orbit and let $\mathbf i,\mathbf j$ be two elements of $O$. The following holds:
\begin{enumerate}
	\item[$(i)$] if $l(\mathbf i)=l(\mathbf j)$ then $\mathbf i=\mathbf j$, 
	\item[$(ii)$] if $r(\mathbf i)=r(\mathbf j)$ then $\mathbf i=\mathbf j$, 
	\item[$(iii)$] we have $r(\mathbf i)=l(\mathbf j)$ if and only if $\mathbf i=R(\mathbf j)$ and $\mathbf j=L(\mathbf i)$. 
\end{enumerate}
\end{lemma}

\begin{proof} We first prove part $(i)$. Since $\mathbf j$ is a permutation of $\mathbf i$, the number of occurrences of $i_k$ in $\mathbf j$ is the same as in $\mathbf i$. Since $l(\mathbf i)=l(\mathbf j)$, the number of occurrences of $i_k$ in $l(\mathbf j)$ is the same as in $l(\mathbf i)$. This is only possible if $j_k=i_k$, and so part $(i)$ is true. The proof of part $(ii)$ is essentially the same as the proof of part $(i)$ above, just left and right have to be exchanged. By definition $r(R(\mathbf j))=l(\mathbf j)$ and $l(L(\mathbf i))=r(\mathbf i)$, so the reverse implication of part $(iii)$ is clearly true. On the other hand if $r(\mathbf i)=l(\mathbf j)$ then $r(\mathbf i)=r(R(\mathbf j))$ by the above, so
$\mathbf i=R(\mathbf j)$ by part $(ii)$, and hence $\mathbf j=L(\mathbf i)$, too. So part $(iii)$ holds, too.
\end{proof}

\begin{defn} 
The {\it truncation set} $\mathbb T(O)$ of an unstable orbit $O$ as  above is its image with respect to the map $l$. 
By part $(iii)$ of Lemma \ref{truncation} this is the same as the image of $O$ with respect to the map $r$. 
By parts $(i)$ and $(ii)$ of Lemma \ref{truncation}, respectively, we get that both $l:O\to\mathbb T(O)$ and $r \colon O\to\mathbb T(O)$ are bijective. 
Let $l^{-1} \colon \mathbb T(O)\to O$ and $r^{-1}\colon \mathbb T(O)\to O$ be the inverse of these two maps, respectively. 
\end{defn}

We are finally ready to prove the following theorem:
\begin{thm}\label{thm:maintheoremfinite} 
Let $V$ be an $\F_2$-vector space and $V_*$ its associated dual algebra. 
Let $B$ be a Boolean ring, $A\subset B$ a finite Boolean subring, and let $A_*$ and $B_*$, respectively, be the associated Boolean graded algebras.
We have:  
\[
\HH^{k,2-k}(V_*\uoplus A_*,V_* \uoplus B_*) = 0 ~ \text{for all} ~ k \ge 3.
\]
Assume $|A|\ge 8$ and let $\mm<0$ be a negative integer. Then we have:  
\[
\HH^{k,\mm}(V_*\uoplus A_*,V_* \uoplus B_*)=0 ~ \text{for all} ~ k \ge 0 ~ \text{with} ~ k \ne 1- \mm.
\]
\end{thm}

\begin{proof} 
For $k \ge 2$ and $j\ge 2$, we are going to show that every cocycle 
\[
f \colon K^k_k(V_*\uoplus A_*)\to B_j
\]
is a coboundary. 
This will imply that $\HH^{k,s}(V_*\uoplus A_*,V_* \uoplus B_*)=0$ for all pairs $(k, \mm)$ with $k\ge 2$ and $\mm \ge 2-k$. 
The remaining cases will be discussed afterwards. 
By the above it will be enough to show the claim orbit by orbit, that is, we may assume without loss of generality that $f$ is supported on an orbit 
$O\subset A_k(I,n)$ of the right translation $R$. We will let $\alpha$ and
$\beta$ denote the head and tail values of $f$, as above.

First assume that $O$ is a stable orbit $\{\mathbf t\}$ with $\mathbf t=t_1,t_2,\ldots,t_k\in I\cup J$. In particular, $t_1=t_k$ and these two elements are in $J$. 
Let $g \colon K^{k-1}_{k-1}(V_*\uoplus A_*)\to B_{j-1}$ be the cochain given by the rule:
$$g(\mathbf s)=\begin{cases}
\alpha(\mathbf t)t_1^{j-1} & \text{, if $\mathbf s=r(\mathbf t)$} ,\\
0 & \text{, otherwise.} \end{cases}$$
\begin{lemma}\label{stable} We have $\partial g=f$.
\end{lemma}

\begin{proof} Note that $t_{k-1}\neq t_k$ since $t_k\in J$ and $\mathbf t$ is admissible. 
Therefore $l(\mathbf t)=t_1,\ldots,t_{k-1}$ is different from $r(\mathbf t)$, and hence
\[
\partial g(\mathbf t)=t_1g(r(\mathbf t))+g(l(\mathbf t))t_k=
\alpha(\mathbf t)t_1^j+0=f(\mathbf t).
\]
Now let $\mathbf u = u_1,u_2,\ldots,u_{k} \in I \cup J$ be an admissible sequence different from $\mathbf t$. 
If $g(r(\mathbf u))\neq 0$, then $r(\mathbf u) = r(\mathbf t)$, and hence $u_1\neq t_1$, as $\mathbf u$ and $\mathbf t$ are different. 
Therefore $u_1g(r(\mathbf u)) = u_1\alpha(\mathbf t)t_1^{j-1}=0$ in this case since $j-1\ge 1$ by assumption, $\alpha(\mathbf t) \in (p(t_1))$ and the product of $u_1$ and $t_1$ is zero when $u_1 \ne t_1$. 
Otherwise, $g(r(\mathbf u))=0$, and hence $u_1g(r(\mathbf u))=0$ trivially in this case. 

If $g(l(\mathbf u)) \neq 0$, then $l(\mathbf u)=r(\mathbf t)$, and so $u_{k-1}=t_k=t_1$. 
So $u_k$ is different from $t_1$, since $\mathbf u$ is admissible. 
Hence we have $g(l(\mathbf u))u_k=\alpha(\mathbf t)t_1^{j-1}u_k=0$ in this case since $j-1\ge 1$ and the product of $t_1$ and $u_k$ is zero when $u_k\ne t_1$. 
Otherwise, $g(l(\mathbf u))=0$, and hence $g(l(\mathbf u))u_k=0$ trivially in this case. 
We get that
\begin{equation*}
\partial g(\mathbf u) = u_1g(r(\mathbf u)) + g(l(\mathbf u)) u_k = 0 + 0 = 0. \qedhere
\end{equation*}
\end{proof}

Now assume that $O$ is an unstable orbit. 
For every sequence $\mathbf t = t_1,t_2,\ldots,t_k\in I\cup J$, 
let $a(\mathbf t), b(\mathbf t)$ denote the first term $t_1$ and the last term $t_k$ of $\mathbf t$, respectively. 
Let $g \colon K^{k-1}_{k-1}(V_*\uoplus A_*)\to B_{j-1}$ be the cochain given by the rule:
\[
g(\mathbf s)=\begin{cases}
\alpha(r^{-1}(\mathbf s))p(a(r^{-1}(\mathbf s)))^{j-1} & \text{, if $\mathbf s\in\mathbb T(O)$} ,\\
0 & \text{, otherwise.} \end{cases}
\]
In particular, $g$ is supported on the truncation set of $O$.

\begin{lemma}\label{other_side} We have:
\[
g(\mathbf s)=\beta(l^{-1}(\mathbf s))p(b(l^{-1}(\mathbf s)))^{j-1}
\]
for every $\mathbf s\in\mathbb T(O)$. 
\end{lemma}
\begin{proof} 
Since $l(l^{-1}(\mathbf s))=\mathbf s=r(r^{-1}(\mathbf s))$, we get that $l^{-1}(\mathbf s)=L(r^{-1}(\mathbf s))$ by part $(iii)$ of Lemma \ref{truncation}. 
Therefore, $b(l^{-1}(\mathbf s))=a(r^{-1}(\mathbf s))$ and $\beta(l^{-1}(\mathbf s))=\alpha(r^{-1}(\mathbf s))$ by Lemma \ref{head/tail_Q}. 
Therefore, 
\begin{equation*}
g(\mathbf s)=\alpha(r^{-1}(\mathbf s))p(a(r^{-1}(\mathbf s)))^{j-1}=
\beta(l^{-1}(\mathbf s))p(b(l^{-1}(\mathbf s)))^{j-1}. \qedhere
\end{equation*}
\end{proof}

\begin{lemma}\label{unstable} 
We have $\partial g=f$.
\end{lemma}

\begin{proof} 
Let $\mathbf u = u_1,u_2,\ldots,u_k\in I\cup J$ be an admissible sequence. 
First assume that $\mathbf u \in O$. 
Then
\[
\partial g(\mathbf u) = u_1g(r(\mathbf u)) + g(l(\mathbf u))u_k 
=
\alpha(\mathbf u)p(u_1)^j+\beta(\mathbf u)p(u_k)^j = f(\mathbf u)
\]
where we use the definition of $g$ and Lemma \ref{other_side} in rewriting the first and second summand, respectively. 
Now assume that $\mathbf u$ is not in $O$. 
If $g(r(\mathbf u)) \neq 0$, then $r(\mathbf u) = r(\mathbf t)$ for some $\mathbf t\in O$. 
Then $u_1\neq t_1$, as $\mathbf u$ and $\mathbf t$ are different when $\mathbf u \notin O$. 
Therefore, $u_1g(r(\mathbf u)) = u_1\alpha(\mathbf t)p(t_1)^{j-1}=0$ in this case since $j-1\ge 1$ by assumption and $u_1 \ne t_1$.  
Otherwise, $g(r(\mathbf u))=0$, and hence $u_1g(r(\mathbf u))=0$ trivially in this case. 

If $g(l(\mathbf u)) \neq 0$, then $l(\mathbf u) = l(\mathbf t)$ for some $\mathbf t\in O$. 
Then $u_k\neq t_k$, as $\mathbf u$ and $\mathbf t$ are different. 
Therefore, since $j-1 \ge 1$ and $t_k \ne u_k$, we have $g(l(\mathbf u))u_k=\beta(\mathbf t)p(t_k)^{j-1}u_k=0$ by Lemma \ref{other_side} in this case. 
Otherwise, $g(l(\mathbf u))=0$, and hence $g(l(\mathbf u))u_k=0$ trivially in this case. We get that
\begin{equation*}
\partial g(\mathbf u) = u_1g(r(\mathbf u)) + g(l(\mathbf u))u_k = 0 + 0 = 0. \qedhere
\end{equation*}
\end{proof}

This finishes the proof for $k\ge 2$ and $\mm \ge 2-k$. 
Since both summands and hence the direct sum $V_j \oplus B_j$ are trivial for negative $j$, there are no nontrivial cochains and hence no nontrivial cocycles for $k+\mm<0$. 
Hence it remains to check the values $k\ge 1$ and $\mm = -k$. 
For $\mm = -k$, the first nontrivial differentials in the cochain complex we use to compute Hochschild cohomology are 
\begin{align*}
\Hom_{\F_2}(K^{k}_{k}(V_* \uoplus A_*), \F_2) \xto{\partial_{k}} \Hom_{\F_2}(K^{k+1}_{k+1}(V_* \uoplus A_*), V_{1} \oplus B_1) \xto{\partial_{k+1}} \ldots 
\end{align*} 
Let $f \colon K^k_k(V_* \uoplus A_*) \to \F_2$ be a cocycle. 
Let ${\mathbf t} = t_{i_1}\otimes \ldots \otimes t_{i_k}$ be an admissible sequence. 
Assuming that the set $I \cup J$ consists of at least three elements, there is a $t_{i_{k+1}}$ different from both $t_{i_1}$ and $t_{i_k}$. 
Then 
\[
t_{i_1}f(t_{i_2}\otimes t_{i_3}\otimes\cdots\otimes t_{i_{k+1}}) +
f(t_{i_1}\otimes t_{i_2}\otimes\cdots\otimes t_{i_k})t_{i_{k+1}}=0 
\]  
is only possible if the coefficients of $t_{i_1}$ and $t_{i_{k+1}}$ in the above sum on the left-hand side, i.e., $f(t_{i_2}\otimes t_{i_3}\otimes\cdots\otimes t_{i_{k+1}})$ and $f(t_{i_1}\otimes t_{i_2}\otimes\cdots\otimes t_{i_k})t$, are both $0$. 
In particular, we must have $f({\mathbf t})=0$. 
Thus there are no nontrivial cocycles in this case.  
This finishes the proof of Theorem \ref{thm:maintheoremfinite}. 
\end{proof}

\begin{rem}
We note that the assumption $|A|\ge 8$ cannot be dropped. 
For example, if $A_* = \F_2[x_1] \uoplus \F_2[x_2]$, then $\HH^{k,-k}(A_*)\ne 0$ for even $k > 0$. 
\end{rem}


\section{Mittag-Leffler functors and transfinite recursion on subrings}\label{section:ML}

In this section we recall and prove some results that we will need to extend the claim of Theorem \ref{thm:maintheoremfinite} to infinite subrings in the next section. 
First, we recall some facts about Mittag-Leffler functors.

\begin{defn}\label{ml-def} 
Let $\mathcal{AB}$ denote the category of abelian groups. 
For every ordinal $\kappa$, by slight abuse of notation we let the same symbol denote the associated small category induced by an ordering. 
Let $F \colon \kappa^{op}\to\mathcal{AB}$ be a functor. We say that the functor $F$ is {\it Mittag-Leffler} if
\begin{enumerate}
\item[(ML1)]\label{ML1} for every $\alpha\in\beta\in\kappa$ the map $F_{\beta}\to F_{\alpha}$ is surjective,
\item[(ML2)]\label{ML2} for every limit ordinal $\alpha\in\kappa$ the map $F_{\alpha}\to
\lim_{\beta\in\alpha}F_{\beta}$ is surjective.
\end{enumerate}
\end{defn} 

\begin{lemma}\label{ml-con} 
If $F \colon \kappa^{op}\to\mathcal{AB}$ is a functor such that
\begin{enumerate}
\item[(ML0)]\label{ML0_lemma} for every $\alpha+1\in\kappa$ the map $F_{\alpha+1}\to F_{\alpha}$ is surjective,
\item[(ML2)]\label{ML2_lemma} for every limit ordinal $\alpha\in\kappa$ the map $F_{\alpha}\to
\lim_{\beta\in\alpha}F_{\beta}$ is surjective.
\end{enumerate}
Then $F$ is Mittag-Leffler. 
\end{lemma}
\begin{proof} We only need to check condition (ML1) of Definition \ref{ml-def} for $F$. Note that we may assume without the loss of generality that
$\alpha=\emptyset$ by replacing $\kappa$ with the unique ordinal isomorphic to the segment $[\alpha,\kappa)$ and $F$ with $F|_{[\alpha,\kappa)}$. 
For every $\alpha\leq\beta$ in $\kappa$, let $\phi_{\beta,\alpha}:F_{\beta}\to F_{\alpha}$ be the transition map, and for every limit ordinal $\beta\in\kappa$, 
let $\phi_{\beta}:F_{\beta}\to\lim_{\alpha\in\beta}F_{\alpha}$ be the limit of the transition maps $\phi_{\beta,\alpha}$. 
Now let $\beta\in\kappa$ and $x\in F_{\emptyset}$ be arbitrary. 
We are going to construct using transfinite recursion for every
$\delta\leq\beta$ an element $x_{\delta}\in F_{\delta}$ with the following properties:
\begin{enumerate}
\item[$(i)$] we have $x_{\emptyset}=x$, 
\item[$(ii)$] for every $\gamma\leq\delta\leq\beta$ we have $\phi_{\delta,\gamma}(x_{\delta})=x_{\gamma}$, 
\item[$(iii)$] for every limit ordinal $\delta\leq\beta$ we have
$\phi_{\delta}(x_{\delta})=\lim_{\gamma\in\delta}x_{\gamma}$.
\end{enumerate}
This is sufficient to conclude that claim, as $\phi_{\beta,\emptyset}(x_{\beta})=
x_{\emptyset}=x$, so $\phi_{\beta,\emptyset}$ is surjective, as $x$ was arbitrary. Also note that the limit $\lim_{\gamma\in\delta}x_{\gamma}$ in $(iii)$ makes sense because of the compatibility in $(ii)$. Let us start the transfinite recursion: for the initial $\delta=\emptyset$ just take $x_{\emptyset}=x$. This satisfies $(i)$, the only non-trivial condition in this case. When $\delta$ is a limit ordinal there is an $x_{\delta}\in F_{\delta}$ such that $\phi_{\delta}(x_{\delta})=
\lim_{\gamma\in\delta}x_{\gamma}$ by condition (ML2). Clearly $\phi_{\delta,\gamma}(x_{\delta})=x_{\gamma}$ for every $\gamma\in\delta$. When
$\delta=\gamma+1$ is a successor ordinal then there is an $x_{\delta}\in F_{\delta}$ such that $\phi_{\delta,\gamma}(x_{\delta})=x_{\gamma}$ by condition (ML0). If $\alpha\in\delta$, then either $\alpha=\gamma$, and hence $\phi_{\delta,\alpha}(x_{\delta})=x_{\alpha}$ by the above, or $\alpha\in\gamma$, and hence
$\phi_{\delta,\alpha}(x_{\delta})=\phi_{\gamma,\alpha}(\phi_{\delta,\gamma}(x_{\delta}))=\phi_{\gamma,\alpha}(x_{\gamma})=x_{\alpha}$ by the induction hypothesis.
\end{proof}

\begin{prop}\label{neeman} 
If $F \colon \kappa^{op}\to\mathcal{AB}$ is a Mittag-Leffler functor then $\underleftarrow{\lim}^nF=0$ for every $n>0$.
\end{prop}
\begin{proof} This is Corollary A.3.14 of \cite{Ne}.
\end{proof}

\begin{lemma}\label{exact-lim} 
Let $F,G,H \colon \kappa^{op}\to\mathcal{AB}$ be functors which form a short exact sequence:
\[
\xymatrix{ 0\ar[r] & F \ar[r] & G \ar[r]
& H \ar[r] & 0.}
\]
If $F$ and $G$ are Mittag--Leffler then $\underleftarrow{\lim}^nH=0$ for every $n>0$.
\end{lemma}
\begin{proof} Since the higher limit functors $\underleftarrow{\lim}^n$ are derived functors, we have a long cohomological exact sequence (see \cite[\S 3.5]{We}):
$$\xymatrix{ \cdots\ar[r] & \underleftarrow{\lim}^nG \ar[r] 
& \underleftarrow{\lim}^nH \ar[r] & \underleftarrow{\lim}^{n+1}F \ar[r] & \cdots.}$$
associated to the short exact sequence above. Now we only need to apply Proposition \ref{neeman}. 
\end{proof}

Next we prove a result on the structure of the subrings of a Boolean ring. 
Let $B$ be a Boolean ring. 
For every subring $A\subset B$ and for every $x\in B$, 
let $A\langle x\rangle$ denote the subring of $B$ generated by $A$ and $x$.  
We assume now that $B$ is infinite and let $A \subset B$ be an infinite subring. 
Let $\kappa=|A|$ be its cardinality and choose a bijection $x \colon \kappa\to A$. 
Fix a subring $A_0\subset A$ such that $|A_0|=8$. 
Using transfinite recursion we can construct a set of subrings 
\begin{align}\label{eq:def_of_A_alpha}
\{A_{\alpha}\subseteq A\mid \alpha\in\kappa\}
\end{align} 
of $A$ such that
\begin{enumerate}
\item[$(a)$] we have $A_{\emptyset}=A_0$,
\item[$(b)$] we have $A_{\alpha+1}=A_{\alpha}\langle x(\alpha)\rangle$ for every
$\alpha\in\kappa$,
\item[$(c)$] we have $\bigcup_{\beta\in\alpha}A_{\beta}=A_{\alpha}$ for every limit ordinal $\alpha\in\kappa$.
\end{enumerate}
For this collection of subrings we have some additional properties:
\begin{lemma}\label{numerically_good} The following holds:
\begin{enumerate}
\item[$(i)$] we have $\bigcup_{\alpha\in\kappa}A_{\alpha}=A$,
\item[$(ii)$] we have $A_{\beta}\subseteq A_{\alpha}$ for every 
$\beta\in\alpha\in\kappa$,
\item[$(iii)$] we have $|A_{\alpha}|<\omega$ for every finite
$\alpha\in\kappa$,
\item[$(iv)$] we have $|A_{\alpha}|\leq|\alpha|$ for every infinite
$\alpha\in\kappa$,
\item[$(v)$] we have $8\leq|A_{\alpha}|<|A|$ for every $\alpha\in\kappa$. 
\end{enumerate}
\end{lemma}
\begin{proof} For every $\beta\in\kappa$ we have $x(\beta)\in
A_{\beta}\langle x(\beta)\rangle=A_{\beta+1}\subseteq\bigcup_{\alpha\in\kappa}A_{\alpha}$, so the latter is $A$, and hence $(i)$ holds. 
We are going to show
$(ii)$ by transfinite induction. 
When $\alpha$ is a limit ordinal then $A_{\beta}
\subseteq\bigcup_{\beta\in\alpha}A_{\beta}=A_{\alpha}$ by condition $(c)$. 
When $\alpha=\gamma+1$ is a successor ordinal then $A_{\gamma}\subseteq
A_{\gamma}\langle x_{\gamma}\rangle=A_{\alpha}$ by $(b)$. 
If $\beta\in\alpha$, then either $\beta=\gamma$, and hence $A_{\beta}\subseteq A_{\alpha}$ by the above, or $\beta\in\gamma$, and hence $A_{\beta}\subseteq A_{\gamma}\subseteq A_{\alpha}$ by the induction hypothesis. 
So $(ii)$ is true. 
Since every finitely generated subring of $A$ is finite, we get $(iii)$ by induction, as every positive finite ordinal is a successor ordinal, so condition $(b)$ applies. 

Now we are going to show $(iv)$ by transfinite induction. When $\alpha=\omega$ or a limit ordinal more generally, we have $|A_{\beta}|\leq|\alpha|$ for every
$\beta\in\alpha$ either by part $(iii)$ (when $\alpha=\omega$) or by the induction hypothesis. 
Therefore 
\[
|A_{\alpha}| = |\bigcup_{\beta\in\alpha}A_{\beta}| \leq |\alpha|^2 = |\alpha|
\]
using condition $(c)$. 
When $\alpha=\beta+1$ is a successor ordinal then by condition $(b)$ we have $A_{\alpha}=A_{\beta}\langle x_{\beta}\rangle$, so the latter is generated by $|\alpha|+1=|\alpha|$ elements, and hence its cardinality is at most  $|\alpha|$. 
Therefore $(iv)$ holds, too. 
For every $\alpha\in\kappa$ we have $A_{\emptyset}\subseteq A_{\alpha}$ by part $(ii)$, so $8=|A_{\emptyset}|\leq|A_{\alpha}|$. 
On the other hand, $|A_{\alpha}|<|A|$ for every $\alpha\in\kappa$ by part $(iii)$, when $\alpha$ is finite, and by part $(iv)$, 
when $\alpha$ is infinite since $|A|=\kappa > \alpha$. 
Therefore $(v)$ is true. 
\end{proof}


\section{Kadeishvili's vanishing for connected sums of dual and Boolean graded algebras - the infinite case}\label{section:Boolean_infinite}

Our next goal is to extend Theorem \ref{thm:maintheoremfinite} to the case of infinite subrings $A \subseteq B$. 
In particular, we will extent the theorem to the case $A=B$.  
This will require a considerable amount of work. 

\begin{notn} 
We fix a vector space $V$ over $\mathbb F_2$ and let $V_*$ denote the associated dual algebra over $\mathbb F_2$. We fix a Boolean ring $B$ and let $B_*$ denote the associated graded Boolean algebra. For every subring $A\subseteq B$ with associated Boolean graded algebra $A_*$ consider the cochain complex:
\[
\xymatrix{ \mathrm{Hom}_{\mathbb F_2}(K^{k}_{k}(V_* \uoplus A_*), \mathbb F_2) \ar[r]^-{\partial_{k}} &
\mathrm{Hom}_{\mathbb F_2}(K^{k+1}_{k+1}(V_* \uoplus A_*), V_1 \oplus B_1)
\ar[r]^-{\partial_{k+1}} & \cdots.}
\]
Let $F^{k,d-k}(A)$ denote the group $\mathrm{Hom}_{\mathbb F_2}(K^{k+d}_{k+d}(V_* \uoplus A_*), (V_*\uoplus B_*)_d)$ of degree $d$ cochains in this complex, and let $G^{k,d-k}(A)\subseteq F^{k,d-k}(A)$ be the subgroup of cocycles. 
\end{notn}

\begin{lemma}\label{zero} 
Let $A\subset B$ be a subring such that $|A|\geq8$. Then $G^{k,-k}(A)=0$ for all $k \ge 1$. 
\end{lemma}
\begin{proof} Since $|A|\geq8$ there is a finite subring $A_0\subseteq A$ such that $|A_0|=8$. Since every finitely generated subring of $A$ is finite, the ring $A$ is the union of its finite subrings containing $A_0$. 
We already know the claim for the latter by the proof of Theorem \ref{thm:maintheoremfinite}. 
So by taking the limit we get the claim for $A$, too.  
\end{proof}

As a consequence we get the following
\begin{cor}\label{zero2} 
Let $A\subset B$ be a subring such that $|A|\geq8$. 
Then $$\mathrm{HH}^{k,-k+\mm}(V_*\uoplus A_*,V_* \uoplus B_*)=0$$
for all $k\geq 1,\mm \le 0$. 
\end{cor}
\begin{proof}
For $\mm=0$, the assertion follows from Lemma \ref{zero}. 
For $\mm<0$, there are no nontrivial maps in $\mathrm{Hom}_{\mathbb F_2}(K^{k+1}_{k+1}(V_* \uoplus A_*), V_\mm \oplus B_\mm)$, 
since both $V_\mm$ and $B_\mm$ are trivial for negative $\mm$.   
\end{proof} 

Recall that, for every subring $A\subset B$ and for every $x\in B$, 
$A\langle x\rangle$ denotes the subring of $B$ generated by $A$ and $x$.

\begin{prop}\label{surjective} 
Let $A\subset B$ be a subring. 
Then for every $x\in B$ the induced map $G^{k,1-k}(A\langle x\rangle)\to G^{k,1-k}(A)$ is surjective for all $k \ge 2$.
\end{prop}
\begin{proof} Let $f\in G^{k,1-k}(A)$ be an arbitrary cocycle. 
We need to show that $f$ has a lift to $G^{k,1-k}(A\langle x\rangle)$. 
Note that for every subring $C\subset B$ every function in
$$\mathrm{Hom}_{\mathbb F_2}(K^{k}_{k}(V_* \uoplus C_*), V_1)
\subseteq
\mathrm{Hom}_{\mathbb F_2}(K^{k}_{k}(V_* \uoplus C_*), V_1 \oplus B_1)$$
is a cocycle, since the multiplication with elements of degree $\ge 1$ in $V_*$ vanishes. 
Moreover, every $\F_2$-vector space homomorphism $K^{k}_{k}(V_* \uoplus A_*) \to V_1$ can be factored through $K^{k}_{k}(V_* \uoplus A_*) \to K^{k}_{k}(V_* \uoplus A\langle x\rangle_*)$ by sending $x$ to zero. 
This implies that the map
$$\mathrm{Hom}_{\mathbb F_2}(K^{k}_{k}(V_* \uoplus A\langle x\rangle_*), V_1)\to\mathrm{Hom}_{\mathbb F_2}(K^{k}_{k}(V_* \uoplus A_*), V_1)$$
is surjective. 
Since every $\F_2$-vector space homomorphism $K^{k}_{k}(V_* \uoplus A_*) \to V_1 \oplus B_1$ can be written as a sum of homomorphisms into $V_1$ and 
$B_1$, respectively, it remains to show that cocycles with values in $B_1$ can be lifted. 

We therefore assume from now on that the cocycle $f\in G^{k,1-k}(A)$ takes values in $B_1$. 
Consider $B_1$ as a $B$-module and write $f=xf+(1-x)f$. 
Note that $xf$ is again a cocycle, since multiplication of $x$ with elements in $V_* \uoplus A_*$ vanishes. 
The definition of the differential then yields that $xf$ is a cocycle. 
That implies that $(1-x)f$ is a cocycle as well. 
Therefore it will be enough to show that $xf$ and $(1-x)f$ have lifts to $G^{k,1-k}(A\langle x\rangle)$. 
Since the roles of $x$ and $1-x$ are symmetrical, we may assume without the loss of generality that $f=xf$. We are going to construct the lift when $A$ is finite first. We will need some preparations to do so. 
For more background on Boolean rings and proofs of the assertions on Boolean rings we refer to appendix \ref{app:Boolean}.

\begin{notn} 
For every Boolean ring $C$, let $\SB(C) = \mathrm{Spec}(C)$ denote the spectrum of $C$. 
For every set $X$, let $\BB(X)$ denote the ring of $\mathbb F_2$-valued functions on $X$. 
Note that every prime ideal in a Boolean ring $C$ is the kernel of a ring homomorphism $C\to\mathbb F_2$, see Proposition \ref{basic} and Theorem \ref{duality}. 
So every element $x\in C$ gives rise to a function in $\BB(\SB(C))$ which furnishes a ring homomorphism $\ff \colon C \to \BB(\SB(C))$. 
This map is an isomorphism by Theorem \ref{sigma}. 
\end{notn} 

\begin{rem}\label{remrem} 
We say that an element $x\in C$ is an {\it atom} if the support of $\ff(x)$ is a one element set, i.e., we have $\ff(x)(p)\neq 0$ for exactly one $p \in \SB(C)$. 
From this perspective the unique orthonormal basis of a finite Boolean ring $C$ is just the collection of all atoms of $C$, see also Corollary \ref{atomic}. 
Therefore, we have a natural bijection between the elements of the unique orthonormal basis of $C$ and the elements of $\SB(C)$. 
We will identify these two sets in all that follows. 
\end{rem}

\begin{defn} 
For every homomorphism $h \colon C\to D$ between Boolean rings let $h^*\colon \SB(D) \to \SB(C)$ be the induced map. 
Let $a \colon A \to A\langle x\rangle$ be the inclusion map. 
Since $a$ is injective, the induced map $a^* \colon \mathbf S(A\langle x\rangle)\to\mathbf S(A)$ is surjective. 
Since every ring homomorphism $A\langle x\rangle\to\mathbb F_2$ is uniquely determined by its restriction to $A$ and $x$, the fibres of $a^*$ have at most two elements. 
Let $s \colon \SB(A) \to \SB(A\langle x\rangle)$ be the section of $a^*$ such that $s(p)$ is the unique element of $(a^*)^{-1}(p)$, if $(a^*)^{-1}(p)$ is a one element set, and $s(p)$ is the unique element of $(a^*)^{-1}(p)$ such that $\mathfrak f(x)(s(p))=1$, otherwise, for every $p \in \SB(A)$.
\end{defn}

\begin{notn} 
Fix a basis $I$ of $V$. For every homomorphism $h \colon C\to D$ between Boolean rings let by slight abuse of notation $h^*$ also denote the map $I\cup\mathbf S(D)\to I\cup\mathbf S(C)$ which is the identity on $I$ and the map induced by $h$ restricted to $\mathbf S(D)$. Similarly let $s$ also denote the map $I\cup\mathbf S(A)\to I\cup\mathbf S(A\langle x\rangle)$ which is the identity on $I$ and the section introduced above, when restricted to $\mathbf S(A)$. Moreover for every sequence $\mathbf t=t_1,t_2,\ldots,t_k\in I\cup\mathbf S(D)$ let $h^*(\mathbf t)$ denote the sequence $h^*(t_1),h^*(t_2),\ldots,h^*(t_k)\in I\cup\mathbf S(C)$, and for every sequence $t_1,t_2,\ldots,t_k\in I\cup\mathbf S(A)$ let $s(\mathbf t)$ denote the sequence $s(t_1),s(t_2),\ldots,s(t_k)\in I\cup\mathbf S(A\langle x\rangle)$. 
\end{notn}

\begin{defn} 
Let $C$ be a subring of $B$. Under the identification in Remark \ref{remrem} above, 
admissible sequences for $C$ are taking values in $I\cup\mathbf S(C)$. A sequence $\mathbf t=t_1,t_2,\ldots,t_k\in I\cup\mathbf S(A\langle x\rangle)$ is {\it $x$-admissible} if it is admissible and $\mathbf t=s(a^*(\mathbf t))$. 
\end{defn}

\begin{lemma}\label{x-ad} 
The following holds:
\begin{enumerate}
\item[$(i)$] if $\mathbf t=t_1,t_2,\ldots,t_k\in I\cup\mathbf S(A\langle x\rangle)$ is $x$-admissible then $a^*(\mathbf t)$ is admissible,
\item[$(ii)$] the maps $a^*$ and $s$ induce a bijection between the set of $x$-admissible sequences in $I\cup\mathbf S(A\langle x\rangle)$ and admissible sequences in $I\cup\mathbf S(A)$. 
\end{enumerate}
\end{lemma} 
\begin{proof} If $a^*(\mathbf t)$ were not admissible, there would be an index $i$ such that $a^*(t_i)=a^*(t_{i+1})$. Since for every $u\in I\cup\mathbf S(A)$ the element $s(u)$ is unique, this means that $t_i=t_{i+1}$ which contradicts that
$\mathbf t$ is admissible. So claim $(i)$ is true. 
By construction $s$ maps admissible sequences in $I\cup\mathbf S(A)$ to $x$-admissible sequences in $I\cup\mathbf S(A\langle x\rangle)$. Therefore $a^*$ and $s$ are maps between these sets, and since they are inverses of each other, we get that claim $(ii)$ holds, too. 
\end{proof}

Now we are going to define a function $f_x\in
\mathrm{Hom}_{\mathbb F_2}(K^{k}_{k}(V_* \uoplus A\langle x\rangle_*),B_1)$ as the unique function such that for every admissible sequence $\mathbf t=t_1,t_2,\ldots,t_k\in I\cup\mathbf S(A\langle x\rangle)$ we have: 
$$f_x(x_{\mathbf t})=\begin{cases}
f(x_{a^*(\mathbf t)})& \text{, if $\mathbf t$ is $x$-admissible,} \\
0 & \text{, otherwise. } \end{cases}$$
For every admissible sequence $\mathbf t=t_1,t_2,\ldots,t_k\in I\cup\mathbf S(A)$ we compute:
$$f_x(x_{\mathbf t})=\sum_{\mathbf u\in(a^*)^{-1}(\mathbf t)}
f_x(x_{\mathbf u})=f_x(x_{s(\mathbf t)})=f(x_{\mathbf t}),$$
where we used the additivity of $f_x$ and part $(ii)$ of Lemma \ref{x-ad} in the first equation, and its definition in the second and third equations. 
Therefore the restriction of $f_x$ onto $K^{k}_{k}(V_* \uoplus A_*)$ is $f$. 
Since the map $a^*$ commutes with translations, and the translates of $x$-admissible sequences are $x$-admissible, in order to show that $f_x$ is a cocycle it will be sufficient to show for every $x$-admissible sequence $\mathbf t=t_1,t_2,\ldots,t_k\in I\cup\mathbf S(A\langle x\rangle)$ that
$$f_x(x_{\mathbf t})\in (t_1,t_k).$$
Note that $xt\in (s(t))$ for every $t\in I\cup\mathbf S(A)$. Therefore
$$f_x(x_{\mathbf t})=f(x_{a^*(\mathbf t)})=xf(x_{a^*(\mathbf t)})
\in x(a^*(t_1),a^*(t_k))\subseteq (xa^*(t_1),xa^*(t_k))\subseteq(t_1,t_k)$$
using that $s(a^*(\mathbf t))=\mathbf t$. This concludes the proof when $A$ is finite.

Since every subring of a Boolean ring is Boolean, and every Boolean ring is the union of the directed set of its finite subrings, we need to show that the lifts we have constructed for finite subrings are compatible in order to prove the claim in general. 
Compatibility means the following: let $A\subseteq\widetilde A$ be a pair of finite subrings of $B$. 
Let $\widetilde f\in G^{k,1-k}(\widetilde A)$ be a cocycle such that $\widetilde f=x\widetilde f$. 
(The latter condition implies that $\widetilde f$ takes values in $B_1$.) 
Let $f\in G^{k,1-k}(A)$ be the restriction of $\widetilde f$ onto $K^{k}_{k}(V_* \uoplus A_*)$. 
Then $f_x$ is the restriction of the similarly constructed lift $\widetilde f_x$ of $\widetilde f$ onto
$K^{k}_{k}(V_* \uoplus A\langle x\rangle_*)$. 

\begin{notn} 
Let $\widetilde a \colon \widetilde A\to\widetilde A\langle x\rangle$ be the inclusion map. Let $\widetilde s \colon I\cup\mathbf S(\widetilde A)\to I\cup\mathbf S(\widetilde A\langle x\rangle)$ be the section of $\widetilde a^*$ defined the same way as $s$ and let the same symbol denote its extension to a map from series in $I\cup\mathbf S(\widetilde A)$ to series in $I\cup\mathbf S(\widetilde A\langle x\rangle)$ by our usual abuse of notation. Let $i:A\to\widetilde A$ and $j:A\langle x\rangle\to\widetilde A\langle x\rangle$ be the inclusion maps. 
\end{notn}

\begin{lemma}\label{x-ad2} 
The following holds:
\begin{enumerate}
\item[$(i)$] if $\mathbf t=t_1,t_2,\ldots,t_k\in I\cup\mathbf S(\widetilde A\langle x\rangle)$ is $x$-admissible then $j^*(\mathbf t)$ is $x$-admissible, too,
\item[$(ii)$] for every admissible $\mathbf t=t_1,t_2,\ldots,t_k\in I\cup
\mathbf S(A)$ the maps $\widetilde a^*$ and $\widetilde s$ induce a bijection between the set of $x$-admissible sequences in $(j^*)^{-1}(s(\mathbf t))$ and admissible sequences in $(i^*)^{-1}(\mathbf t)$. 
\end{enumerate}
\end{lemma} 
\begin{proof} By part $(i)$ of Lemma \ref{x-ad} we already know that $j^*(\mathbf t)$ is admissible. So if  $a^*(\mathbf t)$ were not $x$-admissible, there would be an index $i$ such that $j^*(t_i)\neq s(a^*(j^*(t_i)))$. Then $t_i\neq \widetilde s(\widetilde a^*(t_i))$ either, since $j^*\circ\widetilde s\circ\widetilde a^*=s\circ a^*\circ j^*$, which is a contradiction. So claim $(i)$ holds. Note that
$\widetilde a^*\circ i^*=j^*\circ a^*$ and $\widetilde s\circ j^*=i^*\circ s$. Therefore $\widetilde a^*$ and $\widetilde s$ are maps between these sets, and since they are inverses of each other, we get that claim $(ii)$ holds, too. 
\end{proof}

Now we are ready to show that $f_x$ is the restriction of $\widetilde f_x$. It will be sufficient to check this on $x_{\mathbf t}$ where $\mathbf t=t_1,t_2,\ldots,t_k\in I\cup\mathbf S(A\langle x\rangle)$ is admissible. If $\mathbf t$ is not $x$-admissible, then there is no $x$-admissible sequence in $(j^*)^{-1}(\mathbf t)$ by Lemma \ref{x-ad2}, so
$$\widetilde f_x(x_{\mathbf t})=\sum_{\mathbf u\in(j^*)^{-1}(\mathbf t)}
\widetilde f_x(x_{\mathbf u})=0=f_x(x_{\mathbf t}),$$
where we used the additivity of $\widetilde f_x$ in the first equation, and the definition of $\widetilde f_x$ and $f_x$ in the second and third equations, respectively. If $\mathbf t$ is $x$-admissible, then
\begin{align*}
\widetilde f_x(x_{\mathbf t}) & = \sum_{\mathbf u\in(j^*)^{-1}(\mathbf t)}
\widetilde f_x(x_{\mathbf u})=
\sum_{\substack{ \mathbf u\in(j^*)^{-1}(\mathbf t) \\  
\mathbf u=\widetilde s(\widetilde a^*( \mathbf u)) } }
\widetilde f(x_{\widetilde a^*(\mathbf u)})\\ 
& = \sum_{\mathbf u\in(i^*)^{-1}(a^*(\mathbf t))}\widetilde f(x_{\mathbf u})
=\widetilde f(x_{a^*(\mathbf t)})=f(x_{a^*(\mathbf t)})=f_x(x_{\mathbf t}),
\end{align*}
where we used the additivity of $\widetilde f_x$ in the first equation, the definition of $\widetilde f_x$ in the second equation, part $(ii)$ of Lemma \ref{x-ad2} in the third equation, the additivity of $\widetilde f$ in the fourth equation, the fact that $f$ is the restriction of $\widetilde f$ in the fifth equation, and the definition of $f_x$ in the last equation. 
This finishes the proof of Proposition \ref{surjective}. 
\end{proof}


We can now use Proposition  \ref{surjective} to show the vanishing of Hochschild cohomology groups. 
Let $A\subset B$ be an infinite subring.   
Let $\kappa=|A|$ be its cardinality and choose a bijection $x \colon \kappa\to A$. 
Let $\mathcal{SA}$ denote the category of subrings of $A$ with respect to the inclusions as morphisms. 
By part $(ii)$ of Lemma \ref{numerically_good} we have a functor $\iota \colon \kappa\to\mathcal{SA}$ with $\iota(\alpha)=A_{\alpha}$ for every
$\alpha\in\kappa$. 
Let $F^k,G^k,H^k \colon \kappa^{op}\to\mathcal{AB}$ be the composition of $\iota$ with the contravariant functors:
\[
C\mapsto F^{k-1,1-k}(C),\quad C\mapsto G^{k,1-k}(C),\quad
C\mapsto H^k(C) = \mathrm{HH}^{k,1-k}(V_*\uoplus C_*,V_* \uoplus B_*)
\]
from $\mathcal{SA}$ to $\mathcal{AB}$, respectively. 

\begin{prop}\label{key} 
We have $\underleftarrow{\lim}^nH^k=0$ for every $n>0$.
\end{prop}
\begin{proof} By Lemma \ref{zero} there is a short exact sequence:
\[
\xymatrix{ 0\ar[r] & F^k \ar[r]^-{\partial_{k}} & G^k \ar[r]
& H^k \ar[r] & 0}
\]
of functors. Therefore it will be enough show that both $F^k$ and $G^k$ are 
Mittag-Leffler by Lemma \ref{exact-lim}. Since $F^k$ is just the dual of an inductive system of vector spaces, it is Mittag-Leffler. As we can check that an element of $F^{k,1-k}(C)$ is a cocycle on $(k+2)$-uples of elements of $C$, we have $G^{k,1-k}(A_{\alpha})=\lim_{\beta\in\alpha}G^{k,1-k}(A_{\beta})$
for every limit ordinal $\alpha\in\kappa$, and hence $G^k$ satisfies property (ML2) of Definition \ref{ml-def} .  
By Proposition \ref{surjective} the functor $G^k$ satisfies property (ML0) of Lemma \ref{ml-con}. 
Hence $G^k$ is Mittag-Leffler by Lemma \ref{ml-con}.
\end{proof}


Now we are ready to prove our main 

\begin{thm}\label{main} 
Let $B$ be an infinite Boolean ring. 
Let $A\subseteq B$ be a subring such that $|A|\geq8$. 
Let $\mm<0$ be a negative integer. 
Then 
\[
\HH^{k,\mm}(V_*\uoplus A_*,V_* \uoplus B_*)=0
\]
for every $k\geq0$ such that $k\neq 1-\mm$. 
In particular, we have 
\[
\HH^{n,2-n}(V_* \uoplus B_*)=0 ~\text{for all}~n \ge 3. 
\]
Hence the graded algebra $V_* \uoplus B_*$ is intrinsically formal. 
\end{thm}
\begin{proof} 
We already know that this is true when $A$ is finite by Theorem \ref{thm:maintheoremfinite}, so we may assume without the loss of generality that $A$ is infinite.  
Let $\kappa=|A|$ be its cardinality. 
We prove Theorem \ref{main} by transfinite induction on the cardinality $\kappa$ of $A$. 
This means we may assume that the claim holds for every subring of $A$ of cardinality less than $\kappa$. 
In particular, the claim of the theorem holds for $A_{\alpha}$ for every $\alpha\in\kappa$ by part $(v)$ of Lemma \ref{numerically_good}, 
where $\{A_{\alpha}\subseteq A\mid \alpha\in\kappa\}$ denotes a system of subrings as constructed in \eqref{eq:def_of_A_alpha} in Section \ref{section:ML}. 
Fix an integer $\mm<0$. 
Since the higher limits of identically zero functors vanish, 
all rows of the spectral sequence of Proposition \ref{prop:spectralsequence}:
$$E_2^{p,q} = {\varprojlim_{\alpha}}^p \mathrm{HH}^{q,\mm}(V_*\uoplus (A_{\alpha})_*,V_*\uoplus B_*) \Rightarrow
\mathrm{HH}^{p+q,\mm}(V_*\uoplus A_*,V_*\uoplus B_*)$$
are zero except when $q=1-\mm$. By Proposition \ref{key} all terms of this row except $E_2^{0,1-\mm}={\varprojlim_{\alpha}} \mathrm{HH}^{1-\mm,\mm}(V_*\uoplus (A_{\alpha})_*,V_*\uoplus B_*)$ are zero, too. 
Hence the spectral sequence degenerates and consequently the theorem holds for $A$, too. 
The last assertion follows from Theorem \ref{thm:KadeishvilicritST}. 
\end{proof}


\section{Proof of the main result}\label{section:Main_result}

Let $G$ be a profinite group which is real projective. 
Let $\mathcal X(G)$ denote the conjugacy classes of involutions of $G$ equipped with its natural topology, 
and let $B = C(\Xh(\Gamma),\F_2)$ be the ring of continuous functions from $\mathcal X(G)$ to $\mathbb F_2$ where we equip the latter with the discrete topology. 
Note that $B$ is a Boolean ring. 
Let $B_*$ be the associated Boolean graded algebra.

By Scheiderer's Theorem \ref{cor:scheiderer}, there is a surjective homomorphism of graded $\F_2$-algebras:
\[
\pi \colon H^*(G,\F_2)\to B_*
\]
such that the $n$-th degree component:
\[
\pi_n \colon H^n(G,\F_2)\to B_n
\]
is an isomorphism except perhaps when $n=1$, and it is surjective when $n=1$. 
Let $V\subseteq H^1(G,\mathbb F_2)$ be the kernel of $\pi_1$, and let $V_*$ be the associated dual algebra over $\F_2$. 
By choosing a section
\[
\phi \colon B_1\to H^1(G,\F_2)
\]
of $\pi_1$ we get an isomorphism of graded $\F_2$-algebras:
\[
H^*(G,\F_2)\cong B_* \uoplus V_*.
\]

By Theorem \ref{main}, the connected sum $B_* \uoplus V_*$ satisfies Kadeishvili's vanishing and is intrinsically formal.  
Thus $H^*(G,\F_2)$ satisfies Kadeishvili's vanishing and is intrinsically formal. 
By Corollary \ref{cor:Kadeishvilicrit} this implies that the differential graded algebra $C^*(G,\F_2)$ is formal. 
This finishes the proof of Theorem \ref{thm:mainthmintro}.

Finally, as explained previously, the Galois groups of fields with virtual cohomological dimension one are real projective profinite groups. 
Hence, Proposition \ref{prop:Booleanisquadratic} implies Theorem \ref{thm:intro_Koszul}.


\appendix

\section{Boolean rings}\label{app:Boolean}

In this section we summarise the assertions on Boolean rings we use in Sections \ref{section:def_Boolean_and_dual}, \ref{section:Boolean_finite} and \ref{section:Boolean_infinite}. 

\begin{defn} 
A ring $R$ is called {\it Boolean} if $x^2=x$ for every $x\in R$.
\end{defn}

\begin{examples}\label{ex1.2} 
The field with two elements $\mathbb F_2$ is a Boolean ring. 
In fact, since $x(1-x)=0$ for all $x$ in a Boolean ring, $\F_2$ is the only Boolean integral domain. 
The direct product of Boolean rings is Boolean, and so, for every set $X$, the direct product ring:
\[
\mathbb F_2^X=\prod_{i\in X}\mathbb F_2
\]
is a Boolean ring. 
Now let $X$ be a topological space, and let $\BB(X)$ denote the ring of functions $f \colon X\to\mathbb F_2$ which are continuous with respect to the discrete topology on $\mathbb F_2$. Since the subrings of Boolean rings are Boolean, and $\BB(X)$ is a subring of $\mathbb F_2^X$, we get that $\BB(X)$ is Boolean, too. 
\end{examples}

\begin{prop}\label{basic} 
In a Boolean ring $R$ the following hold:
\begin{enumerate}
\item[$(i)$]  we have $2x=0$ for every $x\in R$,
\item[$(ii)$]  every prime ideal $\mathfrak p$ is maximal, and $R/\mathfrak p$ is the field with two elements,
\item[$(iii)$] we have $(x,y)=(x+y-xy)$ for every $x,y\in R$,
\item[$(iv)$] every finitely generated ideal is principal.
\end{enumerate}
\end{prop}
\begin{proof} Since
\[
2x=(2x)^2=4x^2=4x,
\]
we get that $2x=0$ by subtracting $2x$ from both sides. Now let
$\mathfrak p$ be a prime ideal in $R$. 
Then the quotient $R/\mathfrak p$ is a Boolean ring and an integral domain. 
By Example \ref{ex1.2}, this implies $R/\mathfrak p \cong \F_2$. 
Note that
\[
x(x+y-xy)=x^2+xy-x^2y=x+xy-xy=x.
\]
Hence $x,y\in(x+y-xy)$. 
Since $x+y-xy\in(x,y)$, claim $(iii)$ is clear. Let $I=( x_1,x_2,\ldots,x_n)$ be an finitely generated ideal of $R$. 
Since
\[
I=((x_1,x_2,\ldots,x_{n-1}),x_n),
\]
we may assume by induction on $n$ that $I=(x,y)$ for some $x,y\in R$. The claim now follows from part $(iii)$. 
\end{proof}
\begin{prop}\label{spec} 
The spectrum $\mathrm{Spec}(R)$ of a Boolean ring is compact and totally separated.
\end{prop}
\begin{proof} Since the spectrum of a commutative ring with a unity is compact, the same holds for $\mathrm{Spec}(R)$, too. Recall that a topological space $X$ is totally separated if for any two distinct points 
$x,y\in X$ there exist disjoint open sets $U\subset X$ containing $x$ and $V\subset X$ containing $y$ such that $X$ is the union of 
$U$ and $V$. Now let $\mathfrak p,\mathfrak q\in\mathrm{Spec}(R)$ be two distinct points. Since they are maximal ideals, there is an $x\in R$ such that $x\in\mathfrak p$ and $x\not\in\mathfrak q$. Since $R/\mathfrak p=\mathbb F_2$ by part $(ii)$ of Proposition \ref{basic}, the former is equivalent to $1-x\not\in\mathfrak p$. As usual for every $f\in R$ let $D(f)\subseteq\mathrm{Spec}(R)$ denote the open subset 
\begin{align}\label{eq:def_of_Df}
D(f)=\{\mathfrak p\in\mathrm{Spec}(R)\mid f\not\in\mathfrak p\}.
\end{align}
Then $\mathfrak p\in D(1-x),\mathfrak q\in D(x)$, the intersection $D(x)\cap D(1-x)$ is empty by part $(ii)$ of Proposition \ref{basic}, while the union $D(x)\cup D(1-x)$ is $\mathrm{Spec}(R)$, since if $x,1-x\in\mathfrak m$ for some $\mathfrak m\in\mathrm{Spec}(R)$ then $1\in\mathfrak m$ which is a contradiction. 
\end{proof}

\begin{notn} 
Let $R$ be a Boolean ring. Then for every $a\in R$ the corresponding section of the structure sheaf of $\mathrm{Spec}(R)$ is a continuous function
\[
\sigma(a) \colon \mathrm{Spec}(R) \to \mathbb F_2
\]
by part $(ii)$ of Proposition \ref{basic}, and the furnished map
\[
\sigma \colon R \to \BB(\mathrm{Spec}(R))
\]
is a ring homomorphism. 
\end{notn}

\begin{thm}\label{sigma} 
For every Boolean ring $R$, the map
$\sigma \colon R \to \BB(\mathrm{Spec}(R))$ is an isomorphism. 
\end{thm}
\begin{proof} 
For every $x\in R$ we have $x^n=x$ by induction, so if $x$ is nilpotent, then it is zero. Therefore the nilradical of $R$ is zero, so by Krull's theorem $\sigma$ is injective. 
So we only need to show that $\sigma$ is surjective. 
Let $f \colon \mathrm{Spec}(R)\to\mathbb F_2$ be a continuous function. Then the set
\[
D(f)=\{\mathfrak p\in\mathrm{Spec}(R)\mid f(x)=1\}
\]
is a closed subset, so it is compact. 
But it is also open, so it can be covered by open subsets of the form $D(a)$, where $a\in R$ by the definition of the topology of the spectrum of rings. 
Since $D(f)$ is compact, it can be covered by finitely many such, so
\[
D(f)=D(a_1)\cup D(a_2)\cup\cdots\cup D(a_n)
\]
for some $a_1,a_2,\ldots,a_n\in R$. 
By part $(iv)$ of Proposition \ref{basic}, there is an $a\in R$ such that $(a)=(a_1,\ldots,a_n)$. 
Then
\begin{align*}
D(a)&=\{\mathfrak p\in\mathrm{Spec}(R)\mid a\not\in\mathfrak p\} \\
& = \{\mathfrak p\in\mathrm{Spec}(R)\mid (a_1,\ldots,a_n)\not\subseteq
\mathfrak p\} \\
& = \{\mathfrak p\in\mathrm{Spec}(R)\mid a_i\not\in
\mathfrak p\textrm{ for some $i$}\} \\
& = D(a_1)\cup D(a_2)\cup\cdots\cup D(a_n),\end{align*}
so $D(f)=D(a)$. Since $f$ and $a$ take values in $\mathbb F_2$, we get that $f=a$. 
\end{proof}

\begin{notn}\label{def1.7} 
Let $X$ be a topological space. 
Then for every $p\in X$ the set
\[
\beta(p)=\{ x \in \BB(X)\mid x(p)=0\}
\]
is the kernel of a surjective ring homomorphism $\BB(X)\to\mathbb F_2$, so it is a maximal ideal in $\BB(X)$. 
Consequently, we have an induced map
\[
\beta \colon X\to\mathrm{Spec}(\BB(X)).
\]
For every $x \in \BB(X)$ the pre-image
\[
\beta^{-1}(D(x))=\{p\in X\mid x(p)=1\}
\]
is open. 
Since the sets $\{D(x)\mid x \in \BB(X)\}$ form a sub-basis of
$\mathrm{Spec}(\BB(X))$, we get that 
$\beta \colon X\to \mathrm{Spec}(\BB(X))$ is continuous.  
\end{notn}

\begin{thm}\label{beta} 
When $X$ is compact and totally separated, then
$\beta$ is a homeomorphism. 
\end{thm}

\begin{proof} Since $X$ is compact and $\mathrm{Spec}(\BB(X))$ is Hausdorff, it will be sufficient to show that $\beta$ is a bijection. Let $x,y\in X$ be two distinct points. Since $X$ is totally separated, so there exist disjoint open sets $U\subset X$ containing $x$ and $V\subset X$ containing $y$ such that $X$ is the union of $U$ and $V$. 
Let $f \colon X\to\mathbb F_2$ be the characteristic function of $U$. 
It is in $\mathcal B(X)$ since the complement of $U$ is open, too. Clearly $f\in\beta(y)$, but $f\not\in\beta(x)$, so $\beta$ is injective.

For every ideal $I \triangleleft \BB(X)$ let $Z(I)\subseteq X$ denote the closed subset
\[
Z(I)=\{x\in X\mid f(x)=0\quad(\forall f\in I)\}.
\]
We claim that for every proper ideal $I\triangleleft\BB(X)$ the set $Z(I)$ is non-empty. First consider the case when $I=(f)$ for some $f \in \BB(X)$. Then
\[
Z(I)=\{x\in X\mid f(x)=0\},
\]
so if this set is empty, then $f$ is the identically one function, and hence $I=(1)=\mathcal B(X)$ is not proper, a contradiction. Next consider the case when $I$ is finitely generated; then it is principal by part $(iv)$ of Proposition \ref{basic}, so $Z(I)$ is non-empty by the above. 
Finally, consider the general case. 
Then $Z(I)$ is the intersection of sets of the form $Z(J)$ where $J$ is a finitely generated ideal of $I$. Since the latter collection of sets is closed under finite intersections, and each member is non-empty by the above, we get $Z(I)$ is also non-empty, since $X$ is compact.

Now let $\mathfrak m\triangleleft\BB(X)$ be a maximal ideal. By the above there is an $x\in Z(\mathfrak m)$. Clearly $\beta(x)\supseteq\mathfrak m$, but $\mathfrak m$ is maximal, and hence
$\beta(x)=\mathfrak m$. Therefore $\beta$ is surjective, too. 
\end{proof}

\begin{notn} 
Let $\mathbf{BO}$ denote the category of Boolean rings where morphisms are ring homomorphisms, and let $\mathbf{CTS}$ denote the category of compact, totally separated topological spaces where morphisms are continuous maps. There are two contravariant functors
\[
\BB \colon \mathbf{CTS}\to\mathbf{BO},\quad X\mapsto \BB(X),
\]
which is well-defined as we saw in Examples \ref{ex1.2}, and 
\[
\SB \colon \mathbf{BO}\to\mathbf{CTS},\quad R\mapsto\mathrm{Spec}(R),
\]
well-defined by Proposition \ref{spec}.
\end{notn}

\begin{thm}[Stone duality]\label{duality} 
The functors $\BB$ and $\SB$ are a pair of dualities of categories. 
\end{thm}
\begin{proof} 
By Theorem \ref{sigma}, the map $\sigma$ is a natural isomorphism between the identity of $\mathbf{BO}$ and $\BB \circ \SB$. 
By Theorem \ref{beta}, the map $\beta$ is a natural isomorphism between the identity of $\mathbf{CTS}$ and $\SB \circ \BB$.
\end{proof}

\begin{cor}\label{6.4/2016} 
Let $R$ be a finite Boolean ring. Then the following are equivalent:
\begin{enumerate}
\item[$(i)$] $R$ is finite.
\item[$(ii)$] $R$ is finitely generated as an $\mathbb F_2$-algebra.
\item[$(iii)$] $R$ is Noetherian.
\item[$(iv)$] $R$ is Artinian.
\item[$(v)$] $\mathrm{Spec}(R)$ is finite.
\item[$(vi)$] $R\cong\mathbb F_2^{X}$ for a finite set $X$.
\end{enumerate}
In this case $R\cong\mathbb F_2^{\mathrm{Spec}(R)}$.
\end{cor}
\begin{proof} 
Every Boolean ring is an $\mathbb F_2$-algebra, so if $R$ is finite, then it is finitely generated as an $\mathbb F_2$-algebra, and hence $(i)$ implies $(ii)$. Every finitely generated $\mathbb F_2$-algebra is Noetherian by Hilbert's basis theorem, so $(ii)$ implies $(iii)$. Every Boolean ring is zero dimensional by part $(ii)$ of Proposition \ref{basic}, so if it is Noetherian, it is Artinian by a standard theorem in commutative algebra (see Theorem 8.5 of \cite{AM} on page 90). Therefore $(iii)$ implies $(iv)$. Every Artinian ring is a finite direct product of Artinian local rings (see Theorem 8.7 of \cite{AM} on page 90), so its spectrum is finite. Therefore $(iv)$ implies $(v)$. Now let $R$ be a Boolean ring whose spectrum is finite. Since $\mathrm{Spec}(R)$ is totally separated by Proposition \ref{spec}, it is discrete, and hence $R\cong\mathbb F_2^{\mathrm{Spec}(R)}$ by Theorem \ref{sigma}. 
In particular, $(v)$ implies $(vi)$. If $R\cong\mathbb F_2^{X}$ for a finite set $X$, then $R$ is clearly finite, so $(vi)$ implies $(i)$.
\end{proof}

\begin{notn} 
Let $\mathbf{FBO}$ denote the category of finite Boolean rings where morphisms are ring homomorphisms, and let $\mathbf{FTS}$ denote the category of finite, totally separated topological spaces where morphisms are continuous maps. Note that the latter is the same as the category of finite, discrete topological spaces. There are two restrictions of functors
\[
\BB|_{\mathbf{FTS}} \colon \mathbf{FTS}\to\mathbf{FBO},\quad X\mapsto \BB(X)
\]
and 
\[
\SB|_{\mathbf{FBO}} \colon \mathbf{FBO}\to\mathbf{FTS},\quad R \mapsto\mathrm{Spec}(R),
\]
where the latter is well-defined by Corollary \ref{spec}.
\end{notn}

\begin{thm}\label{fin-duality} 
The functors $\BB|_{\mathbf{FTS}}$ and
$\SB|_{\mathbf{FTS}}$ are a pair of dualities of categories. 
\end{thm}
\begin{proof} 
The restrictions of the natural isomorphisms $\beta$ and $\sigma$ onto $\mathbf{FTS}$ and $\mathbf{FBO}$ are the respective required natural isomorphisms. 
\end{proof}

\begin{defn} 
Let $R$ be a Boolean ring. We say that two elements $x,y\in R$ are orthogonal if $xy=0$. 
We say that an element $a\in R$ is an {\it atom} if it is non-zero and cannot be written as the sum of two non-zero orthogonal elements of $R$. 
The {\it support} of an element $x\in R$ is the subset $D(x)\subseteq\mathrm{Spec}(X)$ introduced in \eqref{eq:def_of_Df}.  
\end{defn}

\begin{lemma}\label{ortho} 
Let $R$ be a Boolean ring. Then the following holds:
\begin{enumerate}
\item[$(i)$] Two elements $x,y\in R$ are orthogonal if and only if the intersection of their supports is empty. 
\item[$(ii)$] A non-zero element $a\in R$ is an atom if and only if its support cannot be written as the disjoint union of two non-empty open and closed subsets. 
\item[$(iii)$] Every pair of different atoms of $R$ are orthogonal to each other. 
\end{enumerate}
\end{lemma}
\begin{proof} Note that the support of the product $xy$ is the intersection of the supports of $x$ and $y$. Since the support of an element of $R$ is empty if and only if it is zero by Theorem \ref{sigma}, claim $(i)$ follows. 
If $a=x+y$ such that $x,y$ are orthogonal and both non-zero, then the support of $a$ is the disjoint union of the supports of $x$ and $y$ by part
$(i)$. 
On the other hand if the support of $a$ is the disjoint union of the non-empty open and closed subsets $X$ and $Y$, there are elements $x,y\in R$ whose support is $X$, $Y$, respectively, by Theorem \ref{sigma}. 
By part $(i)$ these are orthogonal, non-zero and their sum has the same support as $a$. 
So $a=x+y$, and hence claim $(ii)$ is true. 

Let $a,b\in R$ be two atoms whose product is non-zero. 
Then $a=ab+a(1-b)$ and $aba(1-b)=a^2(b-b^2)=0$, so $a(1-b)=0$ by the definition of atoms. 
We get that $a=ab$. The same reasoning for $b$ show that $b=ab$. 
Therefore $a=b$, and hence $(iii)$ holds. 
\end{proof}

\begin{cor}\label{atomic} 
Let $R$ be a finite Boolean ring. 
Then the atoms of $R$ form a natural basis of $R$ whose elements are orthogonal to each other. 
Moreover, every orthogonal basis consists of atoms. 
\end{cor}
\begin{proof} By Corollary \ref{6.4/2016},  the topological space
$\mathrm{Spec}(R)$ is finite, discrete, and $R\cong\mathbb F_2^{\mathrm{Spec}(R)}$. 
Therefore, an element of $R$ is an atom if and only if its support is a point by claim $(ii)$ of Lemma \ref{ortho}. 
These clearly form a basis of $R$ and they are orthogonal by claim $(iii)$ of Lemma \ref{ortho}. 

Now let $e_1,e_2,\ldots,e_n$ be an orthogonal basis. 
We need to show that each $e_i$ is an atom. 
We may assume without the loss of generality that $i=1$. 
Let $x$ be characteristic function of an element of the support of $e_1$.  
Then $x$ is an atom. 
The support of $e_1$ and $e_i$, $i\neq 1$, is disjoint since they are orthogonal. 
Hence $x$ and $e_i$ are orthogonal, so $xe_i=0$. 
Now write $x$ as a linear combination $\sum_{j\in J}e_j$ of the $e_i$. 
Since $0 = xe_i = \sum_{j\in J}e_je_i$ only if $i$ is not in $J$, 
the above argument implies $x=0$ or $x=e_1$. 
The former is not possible, since $x$ is not zero. 
Hence we get $x=e_1$, and $e_1$ is an atom. 
\end{proof}


\section{Proof of Proposition \ref{prop:spectralsequence}}\label{section:appendix}

For lack of a reference known to the authors, we provide a proof of the existence of the spectral sequence claimed in Proposition \ref{prop:spectralsequence}. 
The proof is an adjustment of the construction of a spectral sequence for the composition of a right derived functor with an inductive limit in \cite[Chapter 4]{jensen}. 
Even though Hochschild cohomology is a right derived functor we need to adjust for the fact that we take Hochschild cohomology of a system over varying base rings. 

\begin{proof}[Proof of Proposition \ref{prop:spectralsequence}] 
We recall from Remark \ref{rem:computing_graded_HH} and \eqref{eq:HH_iso_graded_ST} that we can identify the graded Hochschild cohomology of $A$ and $M$ 
with the cohomology 
\begin{align*}
\HH^{n,s}(A,M) = H^n(\uHom_{\F_2}(A^{\otimes *}, M[s]), d^*).
\end{align*} 
To compute the higher inverse limits occurring in the spectral sequence we will use the resolutions of \cite[Chapter 4]{jensen}. 
Let $\{S_{\alpha}, f_{\alpha\beta}\}_{\alpha}$ be a projective system of $\F_2$-vector spaces. 
We denote by $(\Pi^* S_{\alpha}, \delta^*)$ the cochain complex 
defined in \cite[Chapter 4, page 32]{jensen} whose $n$th space $\Pi^n S_{\alpha}$ is the product 
\[
\prod_{\alpha_0 \le \ldots \le \alpha_n} S_{\alpha_0 \ldots \alpha_p} ~ \text{with}~S_{\alpha_0 \ldots \alpha_n} = S_{\alpha_0}
\]
consisting of families $s_{\alpha_0\ldots \alpha_n}$ of elements in $S_{\alpha_0}$ indexed over tuples $\alpha_0 \le \ldots \le \alpha_n$ in $I$. 
These families inherit the structure of $\F_2$-vector spaces with component-wise operations.  
By \cite[Th\'eor\`eme 4.1]{jensen}, the $p$th higher inverse limit of the system $\{S_{\alpha}, f_{\alpha\beta}\}$ equals the $p$th cohomology group of this cochain complex, i.e.,
\[
{\varprojlim_{\alpha}}^p S_{\alpha} = H^p(\Pi^*S_{\alpha}, \delta^*). 
\]

For an inductive system $\{T_{\alpha}, g_{\alpha\beta}\}_{\alpha}$, we denote by $(\Sigma^*T_{\alpha}, \partial^*)$ the chain complex defined in \cite[page 33]{jensen} with 
\[
\Sigma^nT_{\alpha} = \bigoplus_{\alpha_0 \le \ldots \le \alpha_n} T_{\alpha_0 \ldots \alpha_n} ~ \text{with}~T_{\alpha_0 \ldots \alpha_n} = T_{\alpha_0}.
\]
The chain complex $\Sigma^*T_{\alpha}$ is dual to $\Pi^*$ in the sense that 
\begin{align}\label{eq:dualSigmaPi}
\uHom_{\F_2}(\Sigma^*T_{\alpha}, M) \cong \Pi^*\uHom_{\F_2}(T_{\alpha},M).
\end{align}

Now we construct the spectral sequence following \cite[page 34--35]{jensen}. 
We denote the transition maps of the directed system of subrings $A_{\alpha}$ of $A$ by $\iota_{\alpha\beta} \colon A_{\beta} \into A_{\alpha}$. 
We define a double complex $C^{*,*}= \uHom_{\F_2}(\Sigma A_{\alpha}, M)$ by setting $C^{p,q}$ to be
\begin{align*}
C^{p,q} = \uHom_{\F_2} \left( \bigoplus_{\alpha_0 \le \ldots \le \alpha_p} (A_{\alpha_0 \ldots \alpha_p})^{\otimes q}, M \right). 
\end{align*}
The differential in the $p$-direction is the one induced by $\partial^*$. %
For fixed $p$ and $A_{\alpha_0 \ldots \alpha_p} = A_{\alpha_0}$, the differential in the $q$-direction is the differential $d^*$ of the complex $\uHom_{\F_2} ((A_{\alpha_0 \ldots \alpha_p})^{\otimes q}, M)$.

We consider the two spectral sequences that arise from this double complex. 
First let $p$ be the filtration-index. 
Then, using isomorphism \eqref{eq:dualSigmaPi}, we get 
\begin{align*}
{E'_1}^{p,q} & = \prod_{\alpha_0 \le \ldots \le \alpha_p} H^q  \uHom_{\F_2} ((A_{\alpha_0 \ldots \alpha_p})^{\otimes *}, M) \\
 & =  \prod_{\alpha_0 \le \ldots \le \alpha_p} \HH^{q,s}(A_{\alpha_0 \ldots \alpha_p}, M). 
\end{align*}
By \cite[Th\'eor\`eme 4.1]{jensen} applied to the projective system $\{ \HH^{q,s}(A_{\alpha}, M), \HH^{q,s}(f_{\alpha\beta}, M) \}_{\alpha}$, taking cohomology in the $p$-direction gives
\begin{align*}
{E'_2}^{p,q} = H^p \left(\prod_{\alpha_0 \le \ldots \le \alpha_p} \HH^{q,s}(A_{\alpha_0 \ldots \alpha_p}, M) \right)  
= {\varprojlim_{\alpha}}^p \HH^{q,s}(A_{\alpha}, M).
\end{align*}

Now let $q$ be the filtration-index. 
Taking cohomology in the $p$-direction gives
\begin{align*}
{E''_1}^{p,q} = H^p (\uHom_{\F_2}(\Sigma^* A_{\alpha}, M[s]). 
\end{align*}
Regarded as a graded $\F_2$-vector space, $M[s]$ is an injective object in the category of graded $\F_2$-vector spaces. 
Hence taking cohomology commutes with $\uHom_{\F_2}(-, M[s])$. 
By \cite[page 33]{jensen}, the complex $(\Sigma^*T_{\alpha}, \partial^*)$ is acyclic and $H_0(\Sigma^* T_{\alpha}) = \varinjlim_{\alpha} T_{\alpha}$. 
Since direct sums commute with tensor products, we get 
\begin{align*}
{E''_1}^{p,q} = \begin{cases} \uHom_{\F_2}(\varinjlim_{\alpha} (A_{\alpha})^{\otimes q}, M[s]) & \text{if}~ p=0 \\
0 & \text{if}~ p>0.
\end{cases}
\end{align*}
Since inductive limits commute with tensor products, we have 
\[
\uHom_{\F_2}(\varinjlim_{\alpha} (A_{\alpha})^{\otimes q}, M[s]) \cong \uHom_{\F_2}(A^{\otimes q}, M[s]) 
\]
for every $q$. 
Finally, taking cohomology in the direction of $q$ gives 
\begin{align*}
{E''_2}^{p,q} = \begin{cases} H^q(\uHom_{\F_2}(A^{\otimes *}, M[s])) = \HH^{q,s}(A, M) & \text{if}~ p=0 \\
0 & \text{if} ~ p>0. 
\end{cases}
\end{align*}
This spectral sequence degenerates. 
Hence we get that there is a spectral sequence of  the form 
\[
E_2^{p,q} = {\varprojlim_{\alpha}}^p \HH^{q,s}(A_{\alpha}, M) \Rightarrow 
\HH^{p+q,s}(A, M). \qedhere
\]
\end{proof}


\end{document}